\documentclass[12 pt,reqno]{amsart}
 \pdfoutput=1
\usepackage{pdfpages}
\usepackage{amsmath,amsthm,amssymb,amscd}
\usepackage{setspace}
\usepackage{upgreek}

\usepackage{cite}
\usepackage{url}
\usepackage{mathtools}
\usepackage{fullpage}
\usepackage{float}
\usepackage{comment}

\usepackage[bookmarksopen,bookmarksdepth=3]{hyperref}

\usepackage[shortlabels]{enumitem}

\input xy
\xyoption{all}

\usepackage{microtype}

\allowdisplaybreaks[1]

\newtheorem{thm}{Theorem}
\newtheorem{prop}{Proposition}[section]
\newtheorem{lm}[prop]{Lemma}
\newtheorem{cor}[prop]{Corollary}

\theoremstyle{definition}
\newtheorem{dfn}[prop]{Definition}

\theoremstyle{remark}
\newtheorem{rem}[prop]{Remark}
\newtheorem{ex}[prop]{Example}

\DeclareMathOperator{\Ker}{ker}

\DeclareMathOperator{\Id}{Id}

\DeclareMathOperator{\supp}{supp}
\DeclareMathOperator{\im}{im}

\DeclareMathOperator{\rdim}{rdim}

\newcommand{\rarr}{\rightarrow}
\newcommand{\lrarr}{\longrightarrow}
\newcommand{\Rarr}{\Rightarrow}
\newcommand{\Lrarr}{\Longrightarrow}

\newcommand{\R}{\mathbb{R}}
\newcommand{\C}{\mathbb{C}}
\newcommand{\Z}{\mathbb{Z}}

\newcommand{\M}{\mathcal{M}}
\newcommand{\Mbar}{\overline\M}

\renewcommand{\d}{\partial}

\newcommand{\mC}{\mathfrak{C}}

\renewcommand{\Im}{\im}

\newcommand{\A}{\overline{\!A}\mbox{}}

\newcommand{\D}{\mathcal{D}}

\renewcommand{\a}{\alpha}

\newcommand{\G}{\mathcal{G}}
\newcommand{\X}{\mathcal{X}}
\newcommand{\Y}{\mathcal{Y}}
\newcommand{\cZ}{\mathcal{Z}}
\newcommand{\EPG}{\textbf{EPG}}
\newcommand{\Orb}{\textbf{Orb}}
\DeclareMathOperator{\coker}{coker}
\newcommand{\Uc}{\mathcal U}
\newcommand{\LL}{\mathcal L}
\newcommand{\Fdga}{\textbf{Fdga}}
\newcommand{\Orbpr}{\Orb\textsuperscript{pr}}

\newcommand{\Fdg}{\textbf{Fdg}}
\newcommand{\Fu}{\mathfrak{F}}
\newcommand{\Acl}{A_{cl}}
\newcommand{\Ac}{A_c}

\newcommand{\An}{\A_0}

\usepackage[all,cmtip,2cell,color,matrix,arrow]{xy}

\newdir{S}{{}*!/-5em/@{=}}
\newdir{>S}{!/-5em/@{}*@{=>}}
\newdir{T}{{}*!/-6em/@{=}}
\newdir{>T}{!/-6em/@{}*@{=>}}
\newdir{L}{{}*!/-4em/@{=}}
\newdir{>L}{!/-4em/@{}*@{=>}}
\newcommand{\mZ}{\mathcal{Z}}
\newcommand{\mP}{\mathcal{P}}
\definecolor{grayish}{gray}{0.30}
\newcommand{\hd}{\hat{d}}
\renewcommand{\dh}{\d^h}
\newcommand{\dv}{\d^v}
\newcommand{\iv}{i^v}
\newcommand{\ih}{i^h}

\newcommand{\B}[1]{{\mathcal{B}#1}}

\title{Differential forms on orbifolds with corners}


\date{March 2023}

\author[J. Solomon]{Jake P. Solomon}
\address{Institute of Mathematics\\ Hebrew University, Givat Ram\\Jerusalem, 91904, Israel } \email{jake@math.huji.ac.il}
\author[S. Tukachinsky]{Sara B. Tukachinsky}
\address{School of Mathematical Sciences\\ Tel Aviv University\\Tel Aviv, 6997801, Israel }\email{sarabt1@gmail.com}

\begin{document}

\begin{abstract}
Motivated by symplectic geometry, we give a detailed account of differential forms and currents on orbifolds with corners, the pull-back and push-forward operations, and their fundamental properties. We work within the formalism where the category of orbifolds with corners is obtained as a localization of the category of \'etale proper groupoids with corners. Constructions and proofs are formulated in terms of the structure maps of the groupoids, avoiding the use of orbifold charts. The Fr\'echet space of differential forms on an orbifold and the dual space of currents are shown to be independent of which \'etale proper groupoid is chosen to represent the orbifold.
\end{abstract}

\maketitle

\setcounter{tocdepth}{2}
\tableofcontents

\section{Introduction}
\subsection{Motivation and background}\label{ssec:mab}

In this paper, we give a detailed account of differential forms and currents on orbifolds with corners, the pull-back and push-forward operations, and their fundamental properties.
Our main motivation comes from the Fukaya category and open Gromov-Witten theory in symplectic geometry. Although the body of the paper is independent of symplectic geometry, we recall briefly some relevant references which give applications of the techniques developed here. Consider a symplectic manifold $(X,\omega)$ and a Lagrangian submanifold $L \subset X.$ Choose an $\omega$-tame almost complex structure $J$ on $X.$
Let $\M_k(X,L;\beta)$ denote the space of $J$-holomorphic disk maps $(D^2,\partial D^2) \to (X,L)$ with $k$ marked points on $\partial D^2,$ representing the class $\beta \in H_2(X,L)$. A natural compactification is given by the space of open stable maps $\Mbar_k(X,L;\beta).$ There are evaluation maps at each of the marked points $ev_i : \Mbar_k(X,L;\beta) \to L$ for $i = 1,\ldots,k.$ In nice cases, the compactified moduli space $\Mbar_k(X,L;\beta)$ is a smooth orbifold with corners~\cite{Zernik2}. The Fukaya $A_\infty$ algebra of $L$ is constructed by push-forward and pull-back of differential forms along the evaluation maps $ev_i.$ The structure equations of the Fukaya $A_\infty$ algebra are proved using properties of push-forward and pull-back given below~\cite{ST1}. For certain proofs, it is natural to consider push-forward along the forgetful maps $\pi_k : \Mbar_{k}(X,L;\beta) \to \Mbar_{k-1}(X,L;\beta),$ which are not submersions, and thus one is led to consider currents on orbifolds with corners. From Fukaya $A_\infty$ algebras, one can define open Gromov-Witten invariants of Lagrangian submanifolds~\cite{ST2}. The techniques developed here are further used in proving structural equations for open Gromov-Witten invariants~\cite{ST3}. In general, the moduli spaces $\Mbar_k(X,L;\beta)$ are not orbifolds with corners. Nonetheless, we believe the techniques of this paper will be useful in the ongoing efforts to better understand the virtual fundamental class~\cite{Fukaya,FOOOtoricI,FOOOtoricII,FOOOKuranishi,HoferWysockiZehnderbook,HoferWysockiZehnder,HoferWysockiZehnder1,HoferWysockiZehnder3}, in which differential forms on orbifolds with corners and their infinite dimensional generalizations play an important role.

The preceding applications require us to consider orbifolds with corners. On the other hand, we are not aware of a convenient reference for the results of this paper even in the simpler context of closed orbifolds. Indeed, many constructions based on moduli spaces that are closed orbifolds can be carried out at the level of cohomology, so it is not necessary to develop the chain level theory of differential forms in full in that context. It is because of the appearance of moduli spaces that are orbifolds with corners that the Fukaya category and open Gromov-Witten theory require a chain level approach.

We work within the formalism where the weak 2-category of orbifolds with corners is obtained as a localization of the 2-category of \'etale proper groupoids with corners~\cite{Pronk}. An introduction to the parts of this theory needed for the present paper is given in Section~\ref{sec:owc}. The statements and proofs of the paper can be easily adapted to the simpler formalism in which orbifolds with corners form an ordinary category~\cite{Moerdijk}. However, we opt for the $2$-categorical language because it does not add much complexity to the proofs and offers advantages for applications. In particular, morphisms are local in the $2$-categorical approach~\cite{Lerman}. More broadly, the $2$-categorical approach remains as close as possible to the language of smooth Deligne-Mumford stacks that is prevalent in algebraic geometry. Many of our arguments are derived from Behrend's work on the cohomology of stacks~\cite{Behrend}.

\subsection{Statement of results}

To each orbifold with corners $\X,$ we associate a differential graded algebra $A^*(\X)$ of differential forms on $\X.$ We endow $A^*(\X)$ with the structure of a Fr\'echet space compatible with the differential graded algebra structure. To any morphism of orbifolds with corners $f : \X \to \Y,$ we associate a pull-back morphism
\[
f^*: A^*(\Y) \to A^*(\X),
\]
and to any relatively oriented proper submersion $f : \X \to \Y,$ we associate a push-forward morphism
\[
f_*: A^*(\X) \to A^*(\Y).
\]
The pull-back and push-forward morphisms are continuous with respect to the Fr\'echet topology. Our definitions of differential forms, pull-back, and push-forward, for orbifolds with corners generalize the usual ones for manifolds with corners.

The following theorem is our main result. We refer the reader to Section~\ref{sec:oi} for the notions of strongly smooth maps, fiberwise boundary, and relative orientations.
\begin{thm}\label{thm:main}\leavevmode
\begin{enumerate}[label=(\alph*)]
\item\label{nt}
Let $f,g : \X \to \Y$ be morphisms of orbifolds with corners and let $\alpha : f \Rightarrow g$ be a $2$-morphism. Then $f^* = g^*.$ If $f$ and $g$ are relatively oriented proper submersions, and $\alpha$ is relatively oriented, then also $f_* = g_*.$
	\item\label{opushcomp}
		Let $g: \X\to \Y$, $f:\Y\to \cZ,$ be morphisms of orbifolds with corners. Then
\[
g^* \circ f^* = (f\circ g)^*.
\]
If $f,g,$ are relatively oriented proper submersions, then
		\[
		f_*\circ g_*=(f\circ g)_*.
		\]
	\item\label{opushpull}
		Let $f:\X\to \Y$ be a relatively oriented proper submersion of orbifolds with corners, and let $\alpha\in A^*(\Y),$ $\beta\in A^*(\X)$. Then
		\[
		f_*(f^*\alpha\wedge\beta)=\alpha\wedge f_*\beta.
		\]
	\item\label{opushfiberprod}
		Let
		\[
		\xymatrix{
		{\X\times_\cZ \Y}\ar[r]^{\quad p}\ar[d]^{q}&
        {\Y}\ar@2{L->L}[dl]\ar[d]^{g}\\
        {\X}\ar[r]^{f}&\cZ
		}
		\]
		be a weak pull-back diagram of orbifolds with corners, where $f$ and $g$ are strongly smooth and $g$ is a relatively oriented proper submersion. Equip $q$ with the pull-back relative orientation and let $\alpha\in A^*(\Y).$ Then
		\[
		q_*p^*\alpha=f^*g_*\alpha.
		\]
\item\label{thmstokes}
Let $f:\X\to \Y$ be a strongly smooth relatively oriented proper submersion of orbifolds with corners with $\dim \X=s$, and let $\xi\in A^t(\X)$. Then
\[
d (f_*\xi)=f_*(d \xi)+(-1)^{s+t}\big(f\big|_{\dv \X}\big)_*\xi,
\]
where $\dv \X$ is the fiberwise boundary with respect to $f,$ with the induced relative orientation.
\end{enumerate}
\end{thm}
It follows from Theorem~\ref{thm:main}~\ref{nt} and~\ref{opushcomp} that equivalent orbifolds with corners have isomorphic spaces of differential forms and currents. A direct proof of this fact follows from Lemma~\ref{lm:JFK}, which is one of the main steps in the proof of Theorem~\ref{thm:main}. Since the isomorphism between spaces of differential forms on equivalent orbifolds uses the push-forward operation, it is natural to define push-forward for relatively oriented proper morphisms, which include all equivalences between orbifolds.

The paper could be simplified considerably by only defining differential forms on oriented compact orbifolds. Especially the passage from morphisms of compact orbifolds to general proper morphisms requires some effort. Indeed, to define the push-forward operation for proper morphisms, we introduce the notion of differential forms with clean support, and it is necessary to prove their basic properties. For compact orbifolds, clean support coincides with the more familiar notion of compact support.

In Section~\ref{sec:curr}, we define the space of currents on an orbifold with corners $\X$ to be the continuous dual of the topological vector space of compactly supported differential forms $\Ac^*(\X)$ equipped with the weak$^*$ topology. We prove properties of push-forward and pull-back of currents analogous to those of Theorem~\ref{thm:main}.

From Theorem~\ref{thm:main}~\ref{nt} and~\ref{opushcomp}, we obtain Theorem~\ref{thm:funct} below.
Let \Orb{} denote the weak $2$-category of orbifolds, and let \Orbpr{} denote the subcategory of \Orb{} with the same objects and  relatively oriented proper submersions for morphisms. Let \Fdg{} denote the $2$-category where objects are differential graded Fr\'echet spaces, morphisms are continuous maps commuting with the differential and preserving the grading up to a shift, and $2$-morphisms are identities. Let \Fdga{} denote the $2$-category where objects are Fr\'echet differential graded algebras, morphisms are continuous differential graded algebra homomorphisms, and $2$-morphisms are identities.

\begin{thm}\label{thm:funct}
Differential forms give rise to two homomorphisms of weak $2$-categories,
\[
\Fu^* : \Orb \to \Fdga^{op}, \qquad \Fu_* :\Orbpr \to \Fdg,
\]
given on objects $\mathcal{\X}$ of $\Orb{}$ by
\[
\Fu^*(\X) = \Fu_*(\X) = A^*(\X),
\]
and on morphisms $f : \X \to \Y$ in $\Orb$ by
\[
\Fu^*(f) = f^*, \qquad \Fu_*(f) = f_*.
\]
Of course, any $2$-morphism in $\Orb{}$ is mapped by both $\Fu^*$ and $\Fu_*$ to the identity $2$-morphism.
\end{thm}

\subsection{Outline}
In Section~\ref{sec:owc}, we recall the definition of orbifolds with corners and morphisms thereof. We discuss the notions of submersions, proper morphisms, refinements, and relatively oriented morphisms.
We close the section with some remarks on the historical background of orbifolds as singular spaces.
In Section~\ref{sec:forms} we define the algebra of differential forms and the space of differential forms with clean support on an orbifold and construct explicit isomorphisms between the two. We define the operations of pull-back and push-forward along a proper submersion, and show that for a refinement the two operations are inverse to one another.
Section~\ref{sec:thm1} is devoted to a detailed proof of Theorem~\ref{thm:main}.
Section~\ref{sec:curr} defines currents and currents relative to the boundary and proves an analog of Theorem~\ref{thm:main} in that context.

\subsection{Acknowledgements}
We are grateful to P. Giterman, Y. Karshon, O. Kedar, and A. Yuval for helpful conversations. The first author was partially supported by ERC starting grant 337560 and ISF Grant 569/18.
The second author was partially supported by NSF grant DMS-163852, ISF grant 2793/21, and the Colton Foundation.

\section{Orientations and integration}\label{sec:oi}

\subsection{Manifolds with corners}\label{ssec:moc}

For $M$ a manifold with corners, the boundary $\partial M$ is defined as in~\cite[Definition~2.6]{Joyce2}. So, $\d M$ is again a manifold with corners and comes with a natural map $i_M : \d M \to M.$ We say a map is smooth if all its partial derivatives, including one-sided derivatives, exist and are continuous at every point. We say a map is strongly smooth if it is strongly smooth in the sense of~\cite[Definition~2.1(e)]{Joyce3}. We say a smooth map is a submersion if it is locally diffeomorphic to a projection of a direct product of manifolds with corners to one factor.
In particular, a smooth map induces a smooth map of the boundary, which we call the restriction to the boundary, the restriction of a strongly smooth map to the boundary is also strongly smooth, and the restriction of a submersion to the boundary is also a submersion. As usual, diffeomorphisms are smooth maps with a smooth inverse.

\begin{rem}
A map that is smooth in our terminology is called weakly smooth in~\cite[Definition 3.1]{Joyce2} and~\cite[Definition~2.1(a)]{Joyce3}.
Our definition of strongly smooth is identical to~\cite[Definition~2.1(e)]{Joyce3} but corresponds to smooth in~\cite[Definition 3.1]{Joyce2}.
Our notion of submersion corresponds to a horizontal submersion in the terminology of~\cite[Definition 19(a)]{Zernik2}. For strongly smooth maps it coincides with the submersion of~\cite[Definition 3.2(iv)]{Joyce2}.
\end{rem}

Transversality of strongly smooth maps of manifolds with corners is defined as in~\cite[Definition~6.1]{Joyce2}, so any strongly smooth map is tranverse to a strongly smooth submersion.
Fiber products of transverse strongly smooth maps exist by~\cite[Theorem~6.4]{Joyce2}.
Throughout the paper, we assume without further comment that all maps are smooth.

\subsubsection{Vertical and horizontal boundary}\label{sssec:vhbd}

Let $M,N,$ be smooth manifolds with corners, and let $f:M\to N$ be a strongly smooth map.
Following~\cite[Section 4]{Joyce2}, we decompose $\d M$ into horizontal and vertical parts with respect to $f$, as follows.
Denote by $S^1(N)$ the set of all points of depth 1 in $N$.
Recall that $\beta$ is a local boundary component at $y\in N$ if $\beta$ is a connected component of $S^1(N)$ such that $y\in \overline{\beta}$. Thus,
\[
M\prescript{}{f}\times_{i_{N}}^{}\d N
=
\left\{ \big(x,(y,\beta)\big)\,|\, f(x)=y=i_{N}(y,\beta)\right\}.
\]
Let $(x,(y,\beta))\in M\prescript{}{f}\times_{i_{N}}^{}\d N$ and let $V\subset N$ be a neighborhood of $y$ such that there is a boundary defining function $b:V\to [0,\infty)$ at $(y,\beta)$ in the sense of~\cite[Definition 2.14]{Joyce2}. By the strongly smooth assumption, either $b\circ f\equiv 0$ or $b\circ f:f^{-1}(V)\to [0,\infty)$ is a boundary defining function at $x$ for a unique local boundary component. Therefore, the subset
\[
\Xi
:=
\left\{\big(x,(y,\beta)\big)\,|\, b\circ f \mbox{ is boundary defining at }x \right\}
\subset M\prescript{}{f}\times_{i_{N}}^{}\d N
\]
comes with an inclusion
\[
\xi:\Xi\lrarr \d M
\]
and we write
\begin{gather*}
\d M= \dv M\sqcup \dh M,\\
\dh M=\im(\xi),
\quad
\dv M=\d M\setminus \im(\xi).
\end{gather*}
By~\cite[Proposition 4.3]{Joyce2}, $\dv M, \dh M,$ are closed and open in $\d M$ and thus are manifolds with corners. We denote the restrictions of $i_M$ by
\[
\iv_M:\dv M\lrarr M,
\qquad
\ih_M:\dh M\lrarr M.
\]
These come with canonical relative orientations induced by restricting $o^{i_M}$.

\begin{rem}
Our definition of $\dv M$ is $\d_+^f M$ in the notation of~\cite{Joyce2}.
\end{rem}

\begin{rem}
In the special case when $f$ is a strongly smooth proper submersion, $\dv$ recovers the fiberwise boundary, that is, $\dv M = \coprod_{y\in N}\d(f^{-1}(y))$.
\end{rem}

\subsection{Orientation conventions}\label{ssec:orientations}
For a diffeomorphism $f : M \to N$ of oriented manifolds with corners, we define $sgn(f)$ to be $1$ if $f$ preserves orientation and $-1$ otherwise. We use similar notation for isomorphisms of oriented vector spaces.
We follow~\cite[Convention~7.2(a)]{Joyce2} for orienting the boundary and~\cite[Convention~7.2(b)]{Joyce2} for orienting fiber products, as detailed in the following.

To orient boundary, let $M$ be an oriented orbifold with corners.
Let $p\in \d M$ and let $B$ be a basis for $T_{p}\d M$. Let $N\in T_{i_M(p)}M$ be the outward-pointing normal at $p.$ We say $B$ is positive if $N\circ B$ is a positive basis for $T_{i_M(p)}M$.

To orient fiber products, let $M, N,$ and $P,$ be oriented orbifolds with corners. Let $f:M\to N$ and $g:P\to N$ be transverse smooth maps, and consider the following pull-back diagram.
\begin{equation}\label{eq:pbd}
	\xymatrix{
	{M\times_N P}\ar[r]^(.6)p\ar[d]^q&
	{P}\ar[d]^{g}\\
	{M}\ar[r]^{f}&N
	}
\end{equation}
Let $(m,p)\in M\times P$ with $f(m)=g(p)$.
By the transversality assumption,
\[
F:=df_m \oplus - dg_p: T_mM\oplus T_pP\lrarr T_{f(m)}N
\]
is surjective, and by definition of fiber product, there is a natural isomorphism
\[
\psi:T_{(m,p)}(M\times_N P) \stackrel{\sim}{\lrarr}\Ker(F).
\]
So, we have a short exact sequence
\begin{equation}\label{eq:sesor}
0\lrarr T_{(m,p)}(M\times_N P) \stackrel{\psi}{\lrarr} T_mM\oplus T_pP \stackrel{F}{\lrarr} T_{f(m)}N \lrarr 0,
\end{equation}
and splitting gives an isomorphism
\begin{equation}\label{eq:fpor}
\varphi:T_mM\oplus T_pP\stackrel{\sim}{\lrarr}T_{(m,p)}(M\times_N P)\oplus T_{f(m)}N.
\end{equation}
We take the orientation on $T_{(m,p)}(M\times_N P)$ to be the one that makes $sgn(\varphi)=(-1)^{\dim P\cdot \dim N}$.

For a submersion of manifolds with corners $h: Q \to S$ and $y \in S,$ we orient the fiber $h^{-1}(y)$ by identifying it with the fiber product $\{y\} \times_S Q.$

The preceding orientation conventions determine the signs in Proposition~\ref{prop:proppp} below as well as Stokes' theorem, Proposition~\ref{stokes}.

\begin{rem}\label{rem:orient}
Our convention for orientation of fiber products agrees with that of~\cite{FOOO} in case $f$ is a submersion.

\end{rem}

Let $M$ be a manifold with corners and let $\widetilde M \to M$ denote the orientation cover considered as a $\Z/2$ principal bundle. Let $f:M \to N$ be a map of manifolds with corners. The relative orientation bundle of $f$ is the $\Z/2$ principle bundle
\[
\LL_f := Hom_{\Z/2}(\widetilde M,f^*\widetilde N)\to M.
\]
Alternatively, we can view $\LL_f$ as the quotient of the frame bundle of the line bundle
\[
Hom(\det TM, f^* \det TN) \simeq \det TM^\vee \otimes \det TN
\]
by the multiplicative group $\R_{>0}.$
A \textbf{relative orientation} of $f$ is a section of $\LL_f.$ So, orientations on $M$ and $N$ induce a relative orientation of $f$, but $f$ may be relatively oriented even if neither $M$ nor $N$ are orientable. If $N$ is oriented, a relative orientation of $f$ induces an orientation of $M.$ Thus, in the context of \'etale proper groupoids and orbifolds below, we suffice with defining relative orientations for boundaries and fiber products. Absolute orientations can be deduced from them.

Let $g : N \to Q$ be another map of manifolds with corners. Observe that
\[
\LL_{g\circ f} = \LL_f\times_{\Z/2} f^*\LL_g.
\]
So, relative orientations $o^f$ and $o^g$ of $f$ and $g$ respectively induce a relative orientation $o^g \circ o^f$
of $g \circ f.$
Also, a local diffeomorphism $f : M \to N$ comes with a \textbf{canonical relative orientation}, which we denote by $o^f_c.$ At a point $m \in M,$ the canonical relative orientation is given as follows. Let $v_1,\ldots,v_k \in T_mM$ be a basis, let $u_1 = df_m(v_1),\ldots,u_k = df_m(v_k),$ be the corresponding basis of $T_{f(m)}N,$ and let $v_1^*,\ldots,v_k^*,$ be the dual basis. Then,
\[
(o^f_c)_m = [v_1^*\wedge\cdots\wedge v_k^* \otimes u_1\wedge \cdots \wedge u_k].
\]
If $g : N \to Q$ is also a local diffeomorphism, then $o^g_c \circ o^f_c = o^{g \circ f}_c.$
The boundary map $i_M:\d M\to M$ comes with the \textbf{relative boundary orientation} $o^{i_M}$ as follows.
Let $m\in \im(i_M)$ and $\beta$ a local boundary component at $m$. Let $N\in T_mM$ the outward pointing normal, $u_1,\ldots,u_{n-1}\in T_{(m,\beta)}(\d M)$ a basis, and $u_1^*,\ldots, u_{n-1}^*,$ the dual basis. Then
\[
(o^{i_M})_{(m,\beta)} =
[u_1^*\wedge \cdots \wedge u_{n-1}^*\wedge N\wedge u_1\wedge \cdots \wedge u_{n-1}].
\]

Relative orientations pull-back to fiber products as follows. Consider pull-back diagram~\eqref{eq:pbd}. The isomorphism $\varphi$ from~\eqref{eq:fpor} induces an isomorphism
\[
\det(T_mM) \otimes \det (T_pP) \stackrel{\sim}{\lrarr} \det(T_{(m,p)}(M \times_N P))\otimes\det(T_{f(m)}N),
\]
which in turn induces an isomorphism
\[
\det(T_{(m,p)}(M\times_N P))^\vee\otimes\det(T_m M) \stackrel{\sim}{\lrarr}\det(T_{f(m)}N)\otimes\det(T_pP)^\vee.
\]
Composing with the canonical isomorphism
\[
\det(T_{f(m)}N) \otimes \det(T_pP)^\vee \overset{\sim}{\lrarr} \det(T_pP)^\vee \otimes \det(T_{f(m)}N)
\]
given by $y\otimes v \mapsto (-1)^{\dim P \cdot \dim N} v \otimes y,$ we obtain an isomorphism
\begin{equation}\label{eq:switchdetfpor}
\det(T_{(m,p)}(M\times_N P))^\vee\otimes\det(T_m M) \stackrel{\sim}{\lrarr}\det(T_pP)^\vee\otimes\det(T_{f(m)}N).
\end{equation}
Letting $(m,p)$ vary in $M\times_NP,$ we obtain an isomorphism of $\Z/2$ principle bundles
\[
\bar \varphi : \LL_q \to f^*\LL_g.
\]
So, given a relative orientation $o^g$ of $g,$ and letting $f^{-1}o^g$ denote the corresponding section of $f^*\LL_g,$ we define the \textbf{pull-back relative orientation} by
\[
f^*o^g = \bar\varphi^{-1}(f^{-1}o^g).
\]
When the relevant manifolds are oriented, the relative boundary orientation and the pull-back relative orientation agree with the relative orientations induced from the orientations of the boundary and the fiber product respectively. Similarly, in the situation of diagram~\eqref{eq:pbd}, given a relative orientation $o^f$ of $f,$ we can define the \textbf{transpose pull-back relative orientation} $^t\!g^* o^f$ of $p$ so that when $M,N$ and $P,$ are oriented, the relative orientation $^t\! g^*o^f$ is induced by the orientations of $M \times_N P$ and $P.$
\begin{lm}\label{lm:pb&c}
Suppose the map $g$ in the pull-back diagram~\eqref{eq:pbd} is a local diffeomorphism. Then, the map $q$ is also a local diffeomorphism and $o^q_c = f^* o^g_c.$
\end{lm}
\begin{proof}
It is well known that if $g$ is a local diffeomorphism, then so is $q.$ We check the equality of orientations pointwise. Let $r : T_mM \oplus T_pP \to T_{(m,p)}(M \times_N P)$ satisfy $r \circ \psi = \Id.$ Any such $r$ gives rise to an isomorphism $\varphi$ as in~\eqref{eq:fpor} given by
\[
\varphi(u,v) = r(u,v) \oplus F(u,v), \qquad u \in T_{m}M, \quad v \in T_pP.
\]
The main step in the proof is to compute
\[
\det \varphi : \det(T_mM \oplus T_pP) \to \det(T_{(m,p)}(M \times_N P) \oplus T_{f(m)}N).
\]
The orientation does not depend on the choice of $r,$ so we choose $r$ to facilitate the computation.
Since $g$ is a local diffeomorphism, we can define
$
G : T_{m}M \to T_pP
$
by
\[
G(u) = (dg_p)^{-1} \circ df_m(u),
\]
and take $r(u,v) = r(u) = (u,G(u)).$
Let $u_1,\ldots,u_k \in T_mM$ and $v_1,\ldots,v_l \in T_pP$ be bases. Since $g$ is local diffeomorphism, $y_1 = dg_p(v_1),\ldots,y_l = dg_p(v_l) \in T_{f(m)}N$ is a basis. So,
\begin{align}
&\det \varphi(u_1 \wedge \cdots \wedge u_k \wedge v_1 \wedge \cdots \wedge v_l) =  \label{eq:cdetfpor} \\
& \qquad \qquad  = \varphi(u_1) \wedge \cdots \wedge \varphi(u_k) \wedge \varphi(v_1) \wedge \cdots \wedge \varphi(v_l) \notag \\
& \qquad \qquad  = (r(u_1) \oplus df_m(u_1))\wedge \cdots \wedge (r(u_k) \oplus df_m(u_k)) \wedge (-dg_p(v_1)) \wedge \cdots \wedge (-dg_p(v_l)) \notag \\
&  \qquad \qquad = (-1)^lr(u_1) \wedge \cdots \wedge r(u_k) \wedge y_1 \wedge \cdots \wedge y_l. \notag
\end{align}
Observe that $w_1 = r(u_1),\ldots,w_k = r(u_k) \in T_{(m,p)}(M \times_N P)$ is a basis, and let $w_1^*,\ldots,w_k^* \in T_{(m,p)}(M \times_N P)^\vee$ denote the dual basis. Let $v_1^*,\ldots,u_l^* \in T_pP^\vee$ denote the dual basis of $v_1,\ldots,v_l.$
Since
\[
\dim P \cdot \dim N + l = l^2 + l \equiv 0 \pmod 2,
\]
it follows from equation~\eqref{eq:cdetfpor} that the isomorphism~\eqref{eq:switchdetfpor} is given by
\begin{equation}\label{eq:cswitchdetfpor}
w_1^* \wedge \cdots \wedge w_k^* \otimes u_1 \wedge \cdots \wedge u_k  \mapsto  v_1^*\wedge \cdots\wedge v_l^* \otimes y_1 \wedge \cdots \wedge y_l.
\end{equation}
Since $dg_p(v_i) = y_i$ and $dq_{(m,p)}(w_i) = u_i,$ the products $v_1^*\wedge \cdots\wedge v_l^* \otimes y_1 \wedge \cdots \wedge y_l$ and $w_1^* \wedge \cdots \wedge w_k^* \otimes u_1 \wedge \cdots \wedge u_k$ represent the canonical orientations of $g$ and $q$ respectively.
\end{proof}

\subsection{Integration properties}
Let $f: M\to N$ be a relatively oriented proper submersion with fiber dimension $\rdim f =r$. Denote by $f_* : A^*(M) \to A^*(N)[-r]$ the push-forward of forms along $f$, that is, integration over the fiber. We will need the following properties of $f_*$ formulated in~\cite[Section 3.1]{KupfermanSolomon}. 

\begin{prop}\label{prop:proppp}\leavevmode
\begin{enumerate}
	\item \label{normalization}
	Let $M$ be oriented, let $f:M\to pt$ and let $\alpha\in A^m(M)$ with compact support. Then
	\[
	f_*\alpha=\begin{cases}
	\int_M\alpha,& m=\dim M,\\
	0,&\text{otherwise}.
	\end{cases}
	\]
	\item\label{prop:pushcomp}
		Let $g: P\to M$, $f:M\to N,$ be proper submersions with relative orientations $o^g$ and $o^f$ respectively.
Then, equipping $f\circ g$ with the relative orientation $o^f \circ o^g,$ we have
		\[
		f_*\circ g_*=(f\circ g)_*.
		\]
	\item\label{prop:pushpull}
		Let $f:M\to N$ be a relatively oriented proper submersion, and let $\alpha\in A^*(N),$ $\beta\in A^*(M)$. Then
		\[
		f_*(f^*\alpha\wedge\beta)=\alpha\wedge f_*\beta.
		\]
	\item\label{prop:pushfiberprod}
		Let
		\[
		\xymatrix{
		{M\times_N P}\ar[r]^{\quad p}\ar[d]^{q}&
        {P}\ar[d]^{g}\\
        {M}\ar[r]^{f}&N
		}
		\]
		be a pull-back diagram, where $f$ and $g$ are strongly smooth and $g$ is a proper submersion. Let $o^g$ be a relative orientation of $g$ and equip $q$ with the pull-back orientation $f^*o^g.$ Let $\alpha\in A^*(P).$ Then
		\[
		q_*p^*\alpha=f^*g_*\alpha.
		\]
\end{enumerate}
\end{prop}

Properties \eqref{normalization},\eqref{prop:pushpull}, and~\eqref{prop:pushfiberprod}, uniquely determine $f_*.$ Furthermore, we have the following generalization of Stokes' theorem. We have corrected the sign of~\cite[Section 3.1]{KupfermanSolomon}.
\begin{prop}[Stokes' theorem]\label{stokes}
Let $f:M\to N$ be a strongly smooth proper submersion with relative orientation $o^f$ and $\dim M=s$, and let $\xi\in A^t(M)$. Then
\[
d (f_*\xi)=f_*(d \xi)+(-1)^{s+t}\big(f\circ \iv_M \big)_*\xi,
\]
where
$(f\circ \iv_M)_*$ is taken with respect to the induced relative orientation $o^f\circ o^{i_M}$.
\end{prop}

Below are some consequences of the properties of integration above that we will need.

\begin{lm}\label{lm:flipor}
Consider the pull-back diagram~\eqref{eq:pbd} with $f,g,$ strongly smooth and transverse, and let $o^f$ be a relative orientation of $f.$ Consider also the following fiber product diagram.
\[
\xymatrix{
 {P\times_N M} \ar[r]^(.6){\hat{q}}\ar[d]^{\hat{p}}  & M \ar[d]^{f}\\
 P \ar[r]^{g} & N
}
\]
Let $\theta : M \times_N P \to P \times_N M$ be the canonical map, so $\hat p \circ \theta = p.$ Then, the pull-back relative orientation $g^* o^f$ of $\hat p$ and the transpose pull-back relative orientation $^t\!g^* o^f$ of $p$ are related as follows:
\[
g^* o^f \circ o_c^\theta = (-1)^{\rdim f \cdot \rdim g}\cdot\,^t\!g^* o^f.
\]
\end{lm}

\begin{proof}
This follows from an argument similar to the proof of Proposition 7.5(a) of~\cite{Joyce2} given in Remark~7.6(ii).
\end{proof}

We will use the following variation of~\cite[Lemma 5.4]{Solomon2018}.
\begin{lm}\label{lm:pp}
Let $f : M \to M$ be a diffeomorphism and let $\alpha \in A^*(M).$ Considering the push-forward along $f$ with respect to the canonical relative orientation, we have
\[
f^*\alpha = f^{-1}_*\alpha.
\]
\end{lm}
\begin{proof}
We begin by showing that $f_*1 = 1.$ It suffices to show $g^* f_*1 = 1$ for $g : P \to M$ the inclusion of an arbitrary point. Indeed, the fiber product $P\prescript{}{g}\times_{f}^{}M$ is the fiber $f^{-1}(g(P)),$ which consists of a single point with positive orientation. So,  Proposition~\ref{prop:proppp}~\eqref{normalization} and~\eqref{prop:pushfiberprod} imply that $g^*f_*1 = \int_{f^{-1}(g(P))} 1 = 1.$ Therefore, using Proposition~\eqref{prop:proppp}~\eqref{prop:pushcomp} and~\eqref{prop:pushpull}, we obtain
\[
f^*\alpha = f^{-1}_* f_* (f^*\alpha \wedge 1) = f^{-1}_* (\alpha \wedge f_*1) = f^{-1}_* \alpha.
\]
\end{proof}

\begin{lm}\label{lm:flip}
		Let
		\[
		\xymatrix{
		{M\times_N P}\ar[r]^{\quad p}\ar[d]^{q}&
        {P}\ar[d]^{g}\\
        {M}\ar[r]^{f}&N
		}
		\]
		be a pull-back diagram, where $f$ and $g$ are strongly smooth and $f$ is a proper submersion. Let $o^f$ be a relative orientation of $f$ and equip $p$ with the transpose pull-back orientation $^t\!g^*o^f.$ Let $\alpha\in A^*(M).$ Then
		\[
		p_*q^*\alpha=(-1)^{\rdim f \cdot \rdim g} g^*f_*\alpha.
		\]
\end{lm}
\begin{proof}
This follows from Lemmas~\ref{lm:flipor} and~\ref{lm:pp}.
\end{proof}

\section{Orbifolds with corners}\label{sec:owc}
In this section we summarize relevant background on orbifolds.

\subsection{\'Etale proper groupoids with corners}\label{ssec:epgwc}
A \textbf{groupoid} is a small category $\G$ in which each morphism is an isomorphism. A groupoid $\X$ is \textbf{\'etale with corners} of dimension $n$ if the object set $X_0$ and morphism set $X_1$ are manifolds with corners of dimension $n$ and the source and target maps $s,t : X_1 \to X_0$ are local diffeomorphisms. It follows that the composition, inverse and identity maps,
\[
m: X_1  \,_{s}\!\times_{t} X_1 \to X_1, \qquad i: X_1 \to X_1, \qquad e: X_0 \to X_1,
\]
are also local diffeomorphisms.
An \'etale groupoid with corners $\X$ is \textbf{proper} if the map $s \times t : X_1 \to X_0 \times X_0$ is proper.

Some important examples of \'etale proper groupoids with corners arise as follows.
\begin{ex}\label{ex:mwc}
Let $M$ be a manifold with corners. We obtain an \'etale proper groupoid $\X^M$ with objects $X^M_0= M$, morphisms $X^M_1 = M$ and source and target maps $s = t = \Id_M.$
\end{ex}
\begin{ex}\label{ex:cover}
More generally, let $M$ be a manifold with corners and let $\Uc = \{U_\alpha\}_{\alpha\in A}$ be an open cover of $M.$ For $\alpha,\beta \in A,$ let $i_{\alpha\beta} : U_\alpha \cap U_{\beta} \to U_\alpha$ denote the inclusion. One obtains an \'etale proper groupoid $\X^\Uc$ with object and morphism spaces and source and target maps given by
\[
X_0^\Uc =  \coprod_{\alpha \in A} U_\alpha, \qquad X^\Uc_1 = \coprod_{\alpha,\beta \in A} U_\alpha \cap U_\beta, \qquad s = \coprod_{\alpha,\beta \in A} i_{\alpha\beta}, \qquad t = \coprod_{\alpha,\beta \in A} i_{\beta\alpha}.
\]
\end{ex}
\begin{ex}\label{ex:action}
Generalizing Example~\ref{ex:mwc} in a different direction, let $M$ be a manifold with corners and let $G$ be a discrete group. Consider a proper action
\[
\varphi : G \times M \to M.
\]
One obtains an \'etale proper groupoid with corners $\X^\varphi$ with objects $X^\varphi_0 = M,$ morphisms $X^\varphi_1 = G \times M,$ source map given by $s(g,m) = m$ and target map given by $t(g,m) = \varphi(g,m).$ In the special case that $M$ is a point, $G$ is a finite group, and $\varphi$ is the trivial action, we write $\B{G}$ for $\X^\varphi.$ This notation is motivated by Example~\ref{ex:BGorb} below.
\end{ex}

A functor of \'etale proper groupoids with corners $F : \X \to \Y$ is called \textbf{smooth} (resp. \textbf{strongly smooth}) if the underlying maps on objects $F_0 : X_0 \to Y_0$ and on morphisms $F_1 : X_1 \to Y_1$ are smooth (resp. strongly smooth).

\begin{ex}\label{ex:Ff}
Consider the \'etale proper groupoid with corners $\X^\Uc$ associated to a cover $\Uc = \{U_\alpha\}_{\alpha \in A}$ of a manifold with corners $M$ as in Example~\ref{ex:cover}. Let $\Uc' = \{U'_\alpha\}_{\alpha \in A'}$ be a refinement of $\Uc$ and let $f : A' \to A$ be a map such that $U'_{\alpha} \subset U_{f(\alpha)}.$ For $\alpha,\beta \in A',$ let $j^f_\alpha : U'_\alpha \to  U_{f(\alpha)}$ and $j^f_{\alpha \beta}: U_\alpha \cap U_\beta \to U_{f(\alpha)} \cap U_{f(\beta)}$ denote the inclusions. Then, we have a smooth functor $F^f : \X^{\Uc'} \to \X^{\Uc}$ given by
\[
F^f_0 = \coprod_{\alpha \in A'} j^f_\alpha , \qquad F^f_1 = \coprod_{\alpha,\beta \in A'} j^f_{\alpha\beta}.
\]
\end{ex}

Let $F,G: \X \to \Y$ be two smooth functors. A natural transformation $\alpha: F \Rightarrow G$ is called \textbf{smooth} if the underlying map $\alpha : X_0 \to Y_1$ is smooth. We denote by \EPG{} the 2-category of \'etale proper groupoids with corners, with morphisms given by smooth functors and 2-morphisms given by smooth natural transformations. Observe that all $2$-morphisms are in fact isomorphisms.

To an \'etale proper groupoid with corners $\X,$ we associate the topological space $|\X|,$ which is the set of isomorphism classes of objects equipped with the quotient topology induced from~$X_0.$ It follows from the properness of $\X$ that $|\X|$ is Hausdorff; see Lemma~\ref{lm:quothaus} below. To a smooth functor of \'etale proper groupoids with corners $F : \X \to \Y,$ we associate the induced continuous map of isomorphism classes $|F| : |\X| \to |\Y|.$ In Examples \ref{ex:mwc}-\ref{ex:Ff} above,
\[
|\X^M| = |\X^\Uc| = M, \qquad |\X^\varphi| = M/G, \qquad |F^f| = \Id_M.
\]

It is tempting to think of an \'etale proper groupoid $\X$ mainly as a means to study the possibly singular space $|\X|.$ However, as the following examples show, important aspects of the structure of $\X$ cannot be seen through $|\X|.$ Thus, it seems preferable to think of $\X$ as a category with additional topological and smooth structures.
\begin{ex}\label{ex:taction}
As a special case of Example~\ref{ex:action}, let $G$ be a finite group and let $\varphi : G \times M \to M$ be the trivial action. Consider also the \'etale proper groupoid $\X^M$ of Example~\ref{ex:mwc}. Then $|\X^\varphi| = M = |\X^M|.$ However, as will be seen below in Example~\ref{ex:itact}, the \'etale proper groupoids $\X^\varphi$ and $\X^M$ behave quite differently from the perspective of integration when $G$ is non-trivial. Integration over such groupoids plays an important role in applications of Atiyah-Bott fixed point localization~\cite{AtiyahBottLoc} to Gromov-Witten theory~\cite{EllingsrudStrommecubics,EllingsrudStrommeAB,KontsevichLoc,Zernik1}.
\end{ex}
\begin{ex}\label{ex:fb}
Let $\Y$ be the \'etale proper groupoid with corners of Example~\ref{ex:mwc} in the case $M = [0,\infty).$ Let $\mZ$ be the \'etale proper groupoid of Example~\ref{ex:action} in the case where $M = \R, \, G = \Z/2,$ and the action $\varphi : G \times M \to M$ is given by $\varphi(1,x) = -x.$
Then,
\[
|\Y| \simeq [0,\infty) \simeq |\mZ|.
\]
However, there are no two smooth functors $F : \mZ \to \Y$ and $G : \Y \to \mZ$ such that $|F| \circ |G| = \Id_{|\Y|}.$ Indeed, if this were the case, then since the quotient map $Y_0 \to |\Y|$ is a bijection, it would follow that $F_0 \circ G_0 = \Id_{Y_0}.$ On the other hand, we must have $F_0(x) = F_0(-x),$ so the derivative of $F_0$ at the point $0 \in \R = \mZ_0$ vanishes, which contradicts the chain rule.

Similarly, consider the smooth functor $H: \Y \to \R$ given by $H_0(x) = x.$ Then there do not exist smooth functors $G : \Y \to \mZ$ and $H' : \mZ \to \R$ such that $H = H' \circ G.$
In applications such as those in symplectic geometry mentioned in Section~\ref{ssec:mab}, smooth functors like $H$ arise naturally. Thus, $\mZ$ and its higher dimensional generalizations are not sufficiently flexible to serve as local models for boundary and corners even though the associated topological spaces are homeomorphic. See also Example~\ref{ex:fbo} below.
\end{ex}

The notion of a morphism of \'etale proper groupoids is particularly sensitive to the underlying categorical structure. See the introduction of~\cite{Lerman} for a discussion of this point. The following example illustrates some of the subtleties.
\begin{ex}\label{ex:BG}
Let $M$ be a manifold with corners, let $\Uc = \{U_\alpha\}_{\alpha\in A}$ be an open cover of $M,$ and let $\X^\Uc$ be the associated \'etale proper groupoid as in Example~\ref{ex:cover}. Let $\B{G}$ be as in Example~\ref{ex:action}. Then there is a canonical bijection between isomorphism classes of smooth functors $\X^\Uc \to \B{G}$ and isomorphism classes of $G$ bundles over $M$ that are trivial over $U_\alpha$ for $\alpha \in A.$ Thus, although $|\B{G}|$ is a single point, there can be many different maps from a given \'etale proper groupoid to $\B{G}.$ We develop this example further in Example~\ref{ex:BGorb} below.
\end{ex}

We say that a smooth functor of \'etale proper groupoids with corners $F : \X \to \Y$ is \textbf{proper} if the induced map $|F| : |\X| \to |\Y|$ is proper. In particular, we say that $\X$ is \textbf{compact} if $|\X|$ is compact.
We say that a morphism $F : \X \to \Y$ in \EPG{} is a \textbf{submersion} if $F_0$ and hence $F_1$ are submersions.
We say that a morphism $F : \X \to \Y$ in \EPG{} is a \textbf{local diffeomorphism} if $F_0$ and hence also $F_1$ are local diffeomorphisms.
We say that morphisms $F:\X\to \mZ,$ $G:\Y\to\mZ,$ are \textbf{transverse} if $F_0$ and $G_0$ are transverse. In particular, any strongly smooth morphism is transverse to a strongly smooth submersion.

An \textbf{orientation} of an \'etale proper groupoid with corners $\X$ is an orientation on the space of objects $X_0$ and an orientation on the space of morphisms $X_1$ such that the local diffeomorphisms $s,t : X_1 \to X_0$ are orientation preserving.

\begin{ex}\label{ex:fbo}
Let $\Y$ and $\mZ$ be the \'etale proper groupoids with corners from Example~\ref{ex:fb}. Then $\Y$ is orientable but $\mZ$ is not.
\end{ex}

We will also use the notion of a relative orientation of a smooth functor $F : \X \to \Y.$ Recall the notation of Section~\ref{ssec:orientations} for relative orientations of smooth maps of manifolds with corners. A \textbf{relative orientation} of a smooth functor $F : \X \to \Y$ is a pair of relative orientations $o^{F_0}$ and $o^{F_1}$ of $F_0$ and $F_1$ respectively, such that
\begin{equation}\label{eq:relor}
o^s_c \circ o^{F_1} = o^{F_0} \circ o^s_c, \qquad o^t_c \circ o^{F_1} = o^{F_0} \circ o^t_c.
\end{equation}
In the case $F$ is a local diffeomorphism, the canonical relative orientations of $F_0$ and $F_1$ give what we call the \textbf{canonical relative orientation} of $F.$
Let $F,G : \X \to \Y$ be relatively oriented smooth functors and let $\alpha: F \Rightarrow G$ be a smooth natural transformation. Observe that the underlying map $\alpha : X_0 \to Y_1$ satisfies $s \circ \alpha = F$ and $t \circ \alpha = G.$ Thus, a \textbf{relative orientation} of $\alpha$ is a relative orientation $o^\alpha$ of the underlying map $\alpha: X_0 \to Y_1$ such that
\begin{equation}\label{eq:arelor}
o^s_c \circ o^\alpha = o^{F_0}, \qquad o^t_c \circ o^\alpha = o^{G_0}.
\end{equation}
\begin{ex}
Suppose $F$ is a smooth functor and let $o^F_1$ and $o^F_2$ be two relative orientations of $F.$ Let $\alpha : F \Rightarrow F$ be the identity natural transformation. Then $\alpha$ is relatively oriented if and only if $o^F_1 = o^F_2.$
\end{ex}

Recall the discussion of the boundary of manifolds with corners from Section~\ref{ssec:moc}. The \textbf{boundary} of an \'etale proper groupoid with corners $\X$ is the \'etale proper groupoid with corners $\partial \X$ together with a morphism $i_\X : \partial \X \to \X$
defined as follows. The object and morphisms spaces of $\partial \X$ are given by
\[
(\partial X)_j = \partial X_j, \qquad j = 0,1,
\]
and the morphism $i_\X$ is given by $(i_\X)_j = i_{X_j}.$ Observe that a local diffeomorphism of manifolds with corners $f : M \to N$ induces a local diffeomorphism of the boundaries $\partial f : \partial M \to \partial N.$ Thus, the structure maps $s,t,m,i,$ and $e,$ of $\X,$ being local diffeomorphisms, induce structure maps on $\partial\X.$ The \textbf{relative boundary orientation} of $i_\X$ is given by $o^{i_\X} = (o^{i_{X_0}},o^{i_{X_1}}).$

We discuss weak fiber products in $\EPG{}$ in Section~\ref{ssec:fpo}. Some references on \'etale proper groupoids are~\cite{Lerman,Moerdijk}. \'Etale proper groupoids with corners are examples of polyfolds~\cite{HoferWysockiZehnderbook}.

\subsection{The weak 2-category of orbifolds with corners}
A smooth functor of \'etale proper groupoids with corners $F$ is called a \textbf{refinement} if $F$ is an equivalence of categories and a local diffeomorphism.

\begin{ex}\label{ex:refine}
The functor $F^f$ of Example~\ref{ex:Ff} is a refinement.
\end{ex}

Observe that a refinement need not have a smooth weak inverse, so a refinement is generally not an equivalence in the $2$-category \EPG{}. Following Zernik~\cite{Zernik2}, we define the weak 2-category of orbifolds with corners \Orb{} to be the weak 2-category of fractions of \EPG{} with respect to refinements as defined by Pronk~\cite{Pronk}. In particular, in the category \Orb{} a refinement is an equivalence. Zernik~\cite{Zernik2} verifies that refinements satisfy the necessary conditions for the category of fractions to be defined.

More explicitly, we can describe \Orb{} as follows. The objects of \Orb{} are the objects of \EPG{}. A morphism $f : \X \to \Y$ in \Orb{} is a pair of morphisms $R : \X' \to \X$ and $F : \X' \to \Y$ in \EPG{} with $R$ a refinement. We abbreviate $f : \X \overset{R}{\leftarrow} \X' \overset{F}{\rightarrow} \Y$ or simply $f = F|R.$ Let $f : \X \overset{R}{\leftarrow} \X' \overset{F}{\rightarrow} \Y$ and $g : \X \overset{S}{\leftarrow} \X'' \overset{G}{\rightarrow} \Y$ be morphisms $\X \to \Y$ in \Orb{}. A 2-morphism $\alpha : f \Rightarrow g$ in \Orb{} is given by a pair of refinements $T_1 : \X''' \to \X'$ and $T_2 : \X''' \to \X''$ and a pair of 2-morphisms $\alpha_1 : R \circ T_1 \Rightarrow S \circ T_2$ and $\alpha_2 : F \circ T_1 \Rightarrow G \circ T_2$ in \EPG{} as illustrated by the following diagram:
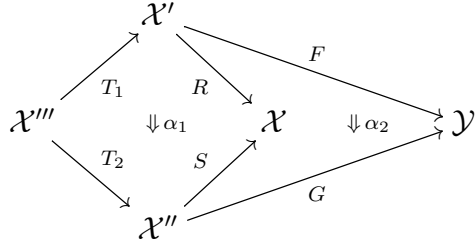
\begin{equation*}
\xymatrix{
& \X' \ar[dr]_R \ar[drrr]^F\\
\X''' \ar[ur]_{T_1} \ar@{}[rr]|(.56){\Downarrow\, \alpha_1} \ar[dr]^{T_2} & & \X \ar@{}[rr]|(.5){\Downarrow\, \alpha_2} && \Y \\
& \X'' \ar[ur]^S \ar[urrr]_G
}
\end{equation*}
A detailed account of compositions of morphisms of different types can be found in~\cite{Pronk}. We discuss the composition of $1$-morphisms in greater detail in Section~\ref{sssec:comp}.

\begin{ex}\label{ex:BGorb}
Let $M$ be a manifold with corners and let $\X^M$ be the associated orbifold as in Example~\ref{ex:mwc}. Let $\B{G}$ be as in Example~\ref{ex:action}. Then isomorphism classes of morphisms $\X^M \to \B{G}$ in \Orb{} are in bijection with $G$ bundles over $M.$ To see this, we combine Examples~\ref{ex:BG} and~\ref{ex:refine}.
\end{ex}

One can identify isomorphic morphisms in $\EPG$ to obtain a (1-)category, which coincides with the category of fractions with respect to refinements as defined by Gabriel-Zisman~\cite{GabrielZisman}. This is the approach taken by~\cite{Moerdijk} to define the category of orbifolds, and is also used in the context of polyfolds~\cite{HoferWysockiZehnderbook}. One disadvantage of this approach is that morphisms in the category of fractions are not local. That is, one cannot construct a morphism by gluing together morphisms on each member of a covering that agree on the intersections. See~\cite{Lerman}. In the weak 2-category approach, one can construct a morphism by gluing together morphisms on each member of a covering if one specifies 2-morphisms on intersections that satisfy a cocycle condition on triple intersections.

Observe that every refinement is a strongly smooth proper submersion. Moreover, a refinement carries a canonical relative orientation since it is a local diffeomorphism. Thus, we say that a morphism $f : \X \to \Y$ in \Orb{} given by a diagram $\X \overset{R}{\leftarrow} \X' \overset{F}{\rightarrow} \Y$
is \textbf{relatively oriented} (resp. \textbf{strongly smooth}, \textbf{proper}, a \textbf{submersion}) if the morphism $F$ in \EPG{} is relatively oriented (resp. strongly smooth, proper, a submersion).
A pair of morphisms $f : \X \to \cZ$ and $g: \Y \to \cZ$ in \Orb{} given by diagrams $\X \overset{R}{\leftarrow} \X' \overset{F}{\rightarrow} \cZ$ and $\Y \overset{S}{\leftarrow} \Y' \overset{G}{\rightarrow} \cZ$ are \textbf{transverse} if $F,G,$ are transverse as maps in $\EPG{}.$ We discuss weak fiber products in $\Orb$ in Section~\ref{ssec:fpo}.

\subsection{Historical context}\label{ssec:hist}

Initially, orbifolds were introduced by Satake~\cite{Satake1} and further developed and used by Thurston~\cite{Thurston} and Haefliger~\cite{Haefliger},
as a form of a singular space. In this approach, orbifolds are defined to be spaces that are locally a quotient of a smooth manifold by a linear action of a finite group.
Differential forms on such orbifolds were defined and studied in~\cite{Satake1,Satake2, Haefliger}. A de Rham theorem was proved.

Differential forms have been defined and studied extensively for other types of singular spaces. For example,
differential forms on possibly singular complex analytic varieties were defined by Grauert and Grothendieck. Although the de Rham theorem does not hold in this context, a substitute was proved by Bloom-Herrera~\cite{BloomHerrera}. Brasselet-Pflaum define Whitney-de Rham cohomology and show it does satisfy the de Rham theorem~\cite{BrasseletPflaum}. Du Bois~\cite{duBois} defines a filtered de Rham complex for separated schemes of finite type over $\C$ and uses it to recover the mixed Hodge structure of Deligne~\cite{Deligne}.
On stratified spaces, forms are defined by Brasselet-Hector-Saralegi~\cite{BrasseletHectorSaralegi}
and Brasselet-Legrand~\cite{BrasseletLegrand}. They prove a de Rham theorem for intersection cohomology.
%
%
Marshall~\cite{Marshall,Marshall2} defines differential forms on subcartesian spaces. Alternative definitions of Zariski and Koszul forms on subcartesian spaces appear in~\cite{WattsMsc, Sniatycki} with a comparison of the three definitions; they are not  equivalent. 
The three types are defined more generally on differential spaces in \'{S}niatycki~\cite{Sniatycki}. Another definition by Kowalczyk~\cite{Kowalczyk} is similar to that of Zariski forms.
Iglesias-Zemmour~\cite{Iglesias} defines diffeological forms, and
Watts~\cite{WattsPhd} verifies that for diffeological orbifolds the definition coincides with that of Satake orbifolds.
%
%
In general, it would be interesting to investigate the relations among the different definitions of differential forms in different settings.

Our approach, however, is from a different perspective. Rather than considering orbifolds as singular spaces, we adopt the approach of~\cite{MoerdijkPronk, Moerdijk, Pronk}
and think of orbifolds in the categorical language of groupoids. A discussion comparing the two definitions is found in~\cite{Moerdijk, AdemLeidaRuan}, and~\cite{HenriquesMetzler, Lerman} stress in addition the importance of the 2-category structure. Notably, the definitions of maps between orbifolds and pullbacks become straightforward in the categorical language, while they are rather subtle in the original approach of Satake~\cite{ChenI, ChenII, ChenRuan}. Moreover, moduli spaces arising in geometry come with natural groupoid structures. We illustrate some advantages of the groupoid approach in Section~\ref{ssec:epgwc}, particularly in Examples~\ref{ex:taction},~\ref{ex:fb},~\ref{ex:BG},~\ref{ex:fbo}, and~\ref{ex:BGorb}. From the point of view of groupoids, an orbifold is a smooth space. That is, the spaces of objects and morphisms are equipped with smooth structures.

A definition of differential forms on differential stacks is given in~\cite{Behrend}, and more specifically for orbifolds from the groupoid perspective in~\cite{AdemLeidaRuan}. This definition,  which is the one used in the present work, is shown to agree with the definition of Satake. Discussions of pull-back, integration, Poincar\'e duality and the de Rham theorem are also given. However, we are not aware of a reference for push-forward of differential forms by a map between orbifolds and the various properties of push-forward developed here. The use of categorical language is of key importance for us to even formulate these properties accurately. Similarly, the categorical approach gives a natural framework for the discussion on currents and their properties, see Section~\ref{sec:curr}.

In algebraic geometry, the analog of an orbifold is a smooth Deligne-Mumford stack. General Deligne-Mumford stacks correspond to \'etale proper groupoids where the spaces of objects and morphisms may be singular. A natural setting for studying singular spaces modeled locally by the zeros of $C^\infty$ functions is algebraic geometry over $C^\infty$ rings, which is developed systematically in~\cite{Joyce4}. A definition of $C^\infty$ stacks is given. These combine categorical structure and singularities. Differential forms on $C^\infty$ rings are defined by Lerman~\cite{Lerman2}. It is not clear in what generality such forms can be integrated over the fiber of a map of $C^\infty$ ringed spaces. So, it is not clear whether the results of the present paper could be extended to that context.

\section{Differential forms on orbifolds with corners}\label{sec:forms}
\subsection{Main definitions}
The present section gives the definitions of differential forms on orbifolds with corners and the push-forward and pull-back operations. The definitions rely on a series of lemmas, which we formulate here and prove in Sections~\ref{ssec:clean} and~\ref{ssec:compare}. Our approach is based on ideas from~\cite{Behrend} and~\cite{Zernik2}.

For a manifold with corners $M,$ denote by $A^*(M)$ the Fr\'echet space of smooth differential forms on $M$ with the $C^\infty$ topology. Note that the exterior derivative $d$ and the pull-back $f^*$ along a smooth map $f$ are linear and continuous. For a proper submersion $f$, the push-forward $f_*$ is linear and continuous as well. Let $\X$ be an object of $\EPG{}.$
The \textbf{differential forms} on $\X$ are given by
\[
A^*(\X) := \ker(s^* - t^* : A^*(X_0) \to A^*(X_1)).
\]
In other words, the differential forms on $\X$ are the invariants of the action of isomorphisms on $A^*(X_0).$
It is a closed subset of $A^*(X_0)$ and thus is also a Fr\'echet space.

\begin{ex}\label{ex:mwcf}
Consider $\X^M$ as in Example~\ref{ex:mwc}. Then, $A^*(\X^M)$ recovers the usual differential forms  $A^*(M).$
\end{ex}
\begin{ex}\label{ex:actionf}
Consider $\X^\varphi$ as in Example~\ref{ex:action}. Then,
$A^*(\X)$ recovers the $G$-invariant differential forms $A^*(M)^G.$
\end{ex}

A subset $C \subset X_0$ is called \textbf{clean} if $t|_{s^{-1}(C)}$ is proper, or equivalently, if $s|_{t^{-1}(C)}$ is proper. A subset $C \subset X_1$ is called clean if it is closed and $s(C),t(C),$ are clean.
\begin{lm}\label{lm:sgc}
Let $C \subset X_1$ be clean. Then $s|_C$ and $t|_C$ are proper.
\end{lm}
\begin{lm}\label{lm:clsc}
A closed subset of a clean subset $C \subset X_i$ is clean for $i = 0,1.$
\end{lm}
\begin{lm}\label{lm:cleanunion}
A finite union of clean subsets of $X_i$ is clean for $i = 0,1.$
\end{lm}

\begin{rem}
If $\X$ is compact, it follows from Lemma~\ref{lm:cpp}
below that a set $C \subset X_i,\; i = 0,1,$ is clean if and only if it is compact. However, when $\X$ is not compact, clean sets are significantly more flexible. For example, Lemma~\ref{lm:pu} below on the existence of partitions of unity holds only for functions with clean support. This leads us to consider differential forms with clean support.
\end{rem}

Let $\Acl^*(X_0), \Acl^*(X_1),$ denote the locally convex subspaces of $A^*(X_0), A^*(X_1),$ consisting of differential forms with clean support on $X_0,X_1,$ respectively. Below, the push-forward of differential forms by any of the structure maps of an \'etale proper groupoid with corners, which are by definition local diffeomorphisms, is always taken with respect to the canonical relative orientation of Section~\ref{ssec:orientations}.
We define the \textbf{differential forms with clean support} on $\X$ by
\[
\Acl^*(\X) := \coker(s_* - t_* : \Acl^*(X_1) \to \Acl^*(X_0)).
\]
Being a quotient of $\Acl^*(X_0)$ by a linear subspace, $\Acl^*(\X)$ is a topological vector space. A priori it is not known to be Hausdorff, but in Corollary~\ref{cor:frech} below we conclude that it is in fact Fr\'echet.
The push-forward operations $s_*,t_*,$ are well defined by Lemma~\ref{lm:sgc}, and they take differential forms with clean support to differential forms with clean support by Lemma~\ref{lm:clsc}. Moreover the difference of two differential forms with clean support is again cleanly supported by Lemma~\ref{lm:cleanunion}. The differential forms with clean support on $\X$ are the coinvariants of the action of isomorphisms on $\Acl^*(X_0).$

For the following two lemmas, let $F : \X \to \Y$ be a morphism in \EPG{}.
\begin{lm}\label{lm:proper}
Let $C_i \subset X_i$ be clean subsets. If $F$ is proper, then the restrictions
\[
F_i|_{C_i} : C_i \to Y_i
\]
are proper maps.
\end{lm}
\begin{lm}\label{lm:pullpush}\mbox{}
\begin{enumerate}
\item\label{item:pull}
The pull-back map $F_0^*: A^*(Y_0) \to A^*(X_0)$ carries $A^*(\Y)$ to $A^*(\X).$
\item\label{item:push}
If $F$ is a relatively oriented proper submersion, the push-forward map
\[
(F_0)_* : \Acl^*(X_0) \to \Acl^*(Y_0)
\]
carries $\Im(s_* - t_*)$ to $\Im(s_*-t_*).$
\end{enumerate}
\end{lm}
In light of Lemma~\ref{lm:pullpush}, we define
\[
F^* : A^*(\Y) \to A^*(\X),
\]
by $F^*\alpha := F_0^*\alpha.$ If $F$ is a relatively oriented proper submersion, we define
\[
F_* : \Acl^*(\X) \to \Acl^*(\Y)
\]
by $F_*[\alpha] := [(F_0)_* \alpha]$ for $\alpha \in \Acl(X_0).$ The operations $F^*$ and $F_*$ are again linear and continuous.
\begin{lm}\label{lm:2mor}
Let $F,G: \X \to \Y$ be morphisms in $\EPG{}$ and suppose there exists a $2$-morphism $\alpha :F \Rightarrow G$. Then $F^* = G^*.$ If $F$ and $G$ are relatively oriented proper submersions, and $\alpha$ is relatively oriented, then $F_* = G_*.$
\end{lm}

Next, for an object $\X$ of \EPG{}, we construct maps between the differential forms $A^*(\X)$ and the cleanly supported differential forms $\Acl^*(\X).$ By the definition of clean, for $\alpha \in \Acl^*(X_0),$ the maps $s|_{\supp(t^*\alpha)}$ and $t|_{\supp(s^*\alpha)}$ are proper, so $s_*t^*\alpha$ and $t_*s^*\alpha$ are well-defined.
\begin{lm}\label{lm:Jwd}\mbox{}
\begin{enumerate}
\item\label{item:tst}
For $\alpha \in \Acl^*(X_0),$ we have $t_*s^*\alpha = s_*t^*\alpha.$
\item\label{item:kernelJ}
For $\alpha \in \Im(s_* - t_* : \Acl^*(X_1) \to \Acl^*(X_0)),$ we have
\[
s_*t^*\alpha = 0.
\]
\item\label{item:imageJ}
For $\alpha \in \Acl^*(X_0),$ we have
\[
s_*t^*\alpha \in A^*(\X).
\]
\end{enumerate}
\end{lm}
Let
\[
J : \Acl^*(\X) \to A^*(\X)
\]
be given by $J([\alpha]) := t_* s^* \alpha,$ which is well-defined by Lemma~\ref{lm:Jwd}. Observe that $J$ is linear and continuous.

A \textbf{partition of unity} for an orbifold $\X$ is a smooth function $\rho : X_0 \to [0,1]$ with clean support such that
\[
t_* s^*\rho = 1.
\]
\begin{lm}\label{lm:pu}
For any object $\X$ of \EPG{} there exists a partition of unity.
\end{lm}
Given a partition of unity $\rho,$ we define a linear continuous map
\[
K : A^*(\X) \to \Acl^*(\X)
\]
by $K(\alpha) := [\rho \alpha].$

\begin{lm}\label{lm:JKinv}
The maps $J$ and $K$ are inverse to one another. In particular, $K$ does not depend on the choice of the partition of unity.
\end{lm}

\begin{cor}\label{cor:frech}
$\Acl^*(\X)$ is continuously isomorphic to $A^*(\X)$ and therefore is Fr\'echet.
\end{cor}

In light of the preceding, it is natural to make the following definitions.
\begin{dfn}\label{dfn:pushepg}
Let $F:\X\to \Y$ be a relatively oriented proper submersion in \EPG{}. We define the \textbf{push-forward} operation
\[
F_*:A^*(\X)\lrarr A^*(\Y)
\]
by
\begin{equation}\label{eq:dfnpush}
F_*:=J\circ F_*\circ K.
\end{equation}
Assume now that $\X$ is oriented, let $pt$ denote the \'etale proper groupoid associated with the point via Example~\ref{ex:mwc}, and let $F : \X \to pt$ be the unique smooth functor. Let $\xi\in A^*(\X)$ such that $\pi(\supp(\xi))\subset |\X|$ is compact. We define
\[
\int_{\X}\xi := F_*\xi.
\]
\end{dfn}
\begin{rem}\label{rem:int}
In the situation of the preceding definition, let $\rho:X_0\to [0,1]$ be a partition of unity on $\X.$ Using $J_{pt}=\Id$ and Proposition~\ref{prop:proppp}\eqref{normalization}, we have
\[
\int_{\X}\xi = F_*\xi = (F_0)_*(K\xi) = (F_0)_*(\rho\xi)= \int_{X_0}\rho\xi.
\]
\end{rem}

\begin{ex}\label{ex:itact}
Consider the \'etale proper groupoid $\X = \X^\varphi$ associated to a trivial group action $\varphi: G \times M \to M$ as in Example~\ref{ex:taction}. Then $A^*(\X) = \Acl(\X) = A^*(M).$ A partition of unity $\rho : X_0 = M \to \R$ is given by the constant function taking the value $1/|G|.$ Let $\xi\in A^*(\X)$ such that $\pi(\supp(\xi))\subset |\X|$ is compact. Then, by Remark~\ref{rem:int} we have
\[
\int_\X \xi = \frac{1}{|G|}\int_{X_0} \xi  = \frac{1}{|G|} \int_{M} \xi.
\]
\end{ex}

To pass from the category \EPG{} to the category \Orb{}, we need the following lemma.
\begin{lm}\label{lm:JFK}
If $F : \X \to \Y$ is a refinement, then $F^*: A^*(\Y)\to A^*(\X)$ is an isomorphism with inverse $F_*$ taken with respect to the canonical relative orientation of $F.$
\end{lm}

The main definition of the paper is the following.
\begin{dfn}\label{dfn:push}
Let $\X$ be an object of \Orb{}. We define the \textbf{differential forms} $A^*(\X)$ the same as for $\X$ considered as an object in \EPG{}. Integration is also defined the same as for $\X$ considered as an object in \EPG{}.
Let $f: \X \to \Y$ be the morphism in \Orb{} given by the diagram
\[
\X \overset{R}{\leftarrow} \X' \overset{F}{\rightarrow} \Y.
\]
We define the \textbf{pull-back} operation
\[
f^* : A^*(\Y) \to A^*(\X)
\]
by
\[
f^*\alpha :=R_* \circ F^*\alpha.
\]
If $f$ is a relatively oriented proper submersion, we define the \textbf{push-forward} operation
\[
f_* : A^*(\X) \to A^*(\Y)
\]
by
\[
f_* \alpha :=  F_* \circ R^*\alpha.
\]
\end{dfn}
Thus defined, the operations $f^*$ and $f_*$ are linear and continuous, being a composition of such.

\begin{ex}
Consider the orbifolds $\Y$ and $\mZ$ from Example~\ref{ex:fb}. By Examples~\ref{ex:mwcf} and~\ref{ex:actionf}, we have $A^0(\Y) = A^0([0,\infty))$ and $A^0(\mZ) = A^0(\R)^{\Z/2}.$
\end{ex}

\subsection{Clean subsets}\label{ssec:clean}
\subsubsection{Clean subsets and properness}
Let $\X$ be an object of~\EPG{}.
\begin{lm}\label{lm:cc}
A clean subset $C \subset X_0$ is closed.
\end{lm}
\begin{proof}
Observe that the image of the identity map $e : X_0 \to X_1$ is closed. Moreover, we have $C = s(e(X_0) \cap t^{-1}(C)).$ Since $s|_{t^{-1}(C)}$ is proper and $e(X_0)$ is closed, it follows that $s|_{t^{-1}(C) \cap e(X_0)}$ is proper. So, its image $C$ is closed.
\end{proof}
\begin{lm}\label{lm:prc}
Let $Z$ be a topological space that is sequential and Hausdorff, let $f : Z \to Y$ be a continuous map, and let $C \subset Z.$ If $f|_C : C \to Y$ is proper, then $C$ is closed.
\end{lm}
\begin{proof}
By definition of a sequential space, $C$ is closed if the limit of each convergence sequence in $C$ belongs to $C.$ Since a sequence together with its limit is a compact set, it suffices to show for each compact subset $K \subset Z$ that $C \cap K$ is closed. Since $Z$ is Hausdorff, $K$ is closed in $Z$, so $K \cap C$ is closed in $C$ and $f|_{C \cap K}$ is proper. So, $C \cap K = (f|_{C \cap K})^{-1}(f(K))$ is compact and therefore closed as desired.
\end{proof}

\begin{lm}\label{lm:seqprop}
Let $f : M \to N$ be a continuous map of metrizable spaces. Then $f$ is proper if and only if for every sequence $p_i \in M$ such that $f(p_i)$ is convergent, possibly after passing to a subsequence, there exists $p \in M$ such that $p_i \to p.$
\end{lm}
\begin{proof}
We first prove the `if' part of the lemma. For $K \subset N$ compact we prove that $f^{-1}(K)$ is compact. Indeed, since $M$ is a metrizable space, it suffices to show that any sequence $p_i \in f^{-1}(K)$ has a convergence subsequence. But $f(p_i) \in K,$ so possibly after passing to a subsequence, $f(p_i)$ is convergent. By assumption, after possibly passing to a subsequence again, there exists $p \in M$ such that $p_i \to p,$ which shows that $f^{-1}(K)$ is compact as desired.

Conversely, assume $f$ is proper and let $p_i \in M$ be a sequence such that $f(p_i)$ converges to $q$. Then, the set $K = \{f(p_i)\}_{i = 1}^\infty \cup \{q\}$ is compact by the definition of convergence, and $p_i \in f^{-1}(K).$ Since $f$ is proper, $f^{-1}(K)$ is compact, and possibly after passing to a subsequence, there exists $p \in f^{-1}(K)$ such that $p_i \to p.$
\end{proof}

\begin{lm}\label{lm:projopen}
The projection $\pi : X_0 \to |\X|$ is an open map.
\end{lm}
\begin{proof}
Indeed, let $U \subset X_0$ be open. Then, since $s$ is a local diffeomorphism and hence an open map, it follows that
\[
\pi^{-1}(\pi(U)) = s(t^{-1}(U))
\]
is open. By the definition of the quotient topology, it follows that $\pi(U)$ is open.
\end{proof}

\begin{lm}\label{lm:quothaus}
The space $|\X|$ is Hausdorff.
\end{lm}
\begin{proof}
The quotient $|\X|$ is defined by the equivalence relation $\im(s\times t)\subset X_0\times X_0$. By properness of $\X$, the relation is closed. It now follows from Lemma~\ref{lm:projopen} that $|\X|$ is Hausdorff.
\end{proof}

\begin{lm}\label{lm:metrizable}
The space $|\X|$ is metrizable.
\end{lm}
\begin{proof}
It follows from Lemma~\ref{lm:projopen} that $|\X|$ inherits the locally compact and second countable properties from $X.$ Since $\X$ is a proper groupoid, $|\X|$ is Hausdorff. As a locally compact Hausdorff space, $|\X|$ is regular. By the Urysohn metrization theorem, a second countable regular Hausdorff space is metrizable.
\end{proof}

\begin{lm}\label{lm:cpp}
A subset $C \subset X_0$ is clean if and only if $\pi|_C : C \to |\X|$ is proper.
\end{lm}
\begin{proof}
First, suppose $C$ is clean. We prove that $\pi|_C$ is proper. Let $x_i \in C$ be a sequence such that $\pi(x_i) \to y \in |\X|.$ By Lemmas~\ref{lm:seqprop} and~\ref{lm:metrizable}, it suffices to prove that possibly after passing to a subsequence, there exists $x \in C$ such that $x_i \to x.$ Indeed, by Lemma~\ref{lm:projopen} we can find a sequence $x_i' \in X_0$ such that $\pi(x_i') = \pi(x_i)$ and $x_i' \to x' \in X_0.$ Since $\pi(x_i') = \pi(x_i),$ we can choose a sequence $z_i \in X_1$ such that $s(z_i) = x_i$ and $t(z_i) = x_i'.$ Again invoking Lemma~\ref{lm:seqprop}, since $t|_{s^{-1}(C)}$ is proper, possibly after passing to a subsequence, there exists $z \in X_1$ such that $z_i \to z.$ Thus, choosing $x = s(z),$ continuity of $s$ implies $x_i \to x.$ It follows from Lemma~\ref{lm:cc} that $x \in C$ as desired.

Conversely, suppose $\pi|_C : C \to |\X|$ is proper and let $K \subset X_0$ be compact. We show that $s^{-1}(C) \cap t^{-1}(K)$ is compact. Indeed,
\[
s(s^{-1}(C) \cap t^{-1}(K)) = C \cap s(t^{-1}(K)) = C \cap \pi^{-1}(\pi(K)),
\]
which is compact by properness of $\pi|_C.$ On the other hand, since $\pi|_C$ is proper, $\pi(C)$ is closed, and by continuity $\pi^{-1}(\pi(C))$ is also closed. Thus,
\[
t(s^{-1}(C) \cap t^{-1}(K)) = K \cap t(s^{-1}(C)) = K \cap \pi^{-1}(\pi(C))
\]
is compact as a closed subset of a compact set. Since the product of compact sets is compact, it follows that $(s\times t)(s^{-1}(C) \cap t^{-1}(K))$ is compact. Since $s\times t : X_1 \to X_0 \times X_0$ is proper by the definition of \EPG, we conclude that  $(s\times t)^{-1}((s\times t)(s^{-1}(C) \cap t^{-1}(K))$ is compact. Lemma~\ref{lm:prc} implies that $C$ is closed, and $K$ is closed because it is compact, so $s^{-1}(C) \cap t^{-1}(K)$ is closed by continuity of $s$ and $t.$ Since
\[
s^{-1}(C) \cap t^{-1}(K) \subset (s\times t)^{-1}((s\times t)(s^{-1}(C) \cap t^{-1}(K)),
\]
it follows that $s^{-1}(C) \cap t^{-1}(K)$ is compact as a closed subset of a compact set.
\end{proof}
\begin{rem}
One can prove Lemma~\ref{lm:cpp} without using metrizability of $X_0,X_1,|\X|,$ but instead using the local compactness of $X_0.$ However, given the potential interest in generalizing the results of this paper to infinite dimensional \'etale proper groupoids such as polyfolds, it seemed preferable to avoid using local compactness in an essential way. For polyfolds, at least when the orbit space is paracompact, the metrizability assumptions hold. See Theorems~2.2 and~7.2 of~\cite{HoferWysockiZehnderbook}.
\end{rem}

\begin{proof}[Proof of Lemma~\ref{lm:sgc}]
Since $C$ is clean, $s(C)$ is clean, so $t|_{s^{-1}(s(C))}$ is proper. Since $C$ is closed by definition, also $t|_{s^{-1}(s(C))\cap C}$ is proper. But $s^{-1}(s(C))\cap C = C,$ so $t|_C$ is proper. A similar argument shows that $s|_C$ is proper.
\end{proof}

\begin{proof}[Proof of Lemma~\ref{lm:clsc}]
For $i = 0,$ this follows from the definition since the restriction of a proper map to a closed set is proper. We turn to the case $i = 1.$  Let $D \subset C \subset X_1$ be a closed subset. By Lemma~\ref{lm:sgc}, it follows that $s(D),t(D),$ are closed. Since $s(D) \subset s(C)$ and $t(D) \subset t(C),$ the case $i = 0$ implies that $s(D),t(D),$ are clean, which means that $D$ is clean as desired.
\end{proof}
\begin{proof}[Proof of Lemma~\ref{lm:cleanunion}]
In the case $i = 0,$ this follows from the fact that if the restriction of a map to each of a finite collection of sets is proper, then so is its restriction to their union. The case $i = 1$ follows from the definition and the case $i = 0.$
\end{proof}

\subsubsection{Clean subsets and morphisms}\label{sssec:csm}
\begin{proof}[Proof of Lemma~\ref{lm:proper}]
First we prove that $F_0|_{C_0} : C_0 \to Y_0$ is proper. Since $Y_0$ is Hausdorff, a compact subset $K \subset Y_0$ is closed. So, it suffices to prove that the closed subset
\[
(F_0|_{C_0})^{-1}(K)\subset C_0
\]
is contained in a compact set. This follows from the properness of $|F|,$ Lemma~\ref{lm:cpp}, and the commutativity of the following diagram:
\[
\xymatrix{
C_0 \ar[r]^{F_0|_{C_0}} \ar[d]^{\pi|_{C_0}} & Y_0 \ar[d]^{\pi} \\
|\X| \ar[r]^{|F|} & |\Y|.
}
\]
To prove that $F_1|_{C_1} : C_1 \to Y_1$ is proper, we consider the following diagram:
\[
\xymatrix{
C_1 \ar[d]^{s\times t|_{C_1}} \ar[rrr]^{F_1|_{C_1}} &&& Y_1 \ar[d]^{s\times t} \\
s(C_1) \times t(C_1) \ar[rrr]^(.55){F_0|_{s(C_1)} \times F_0|_{t(C_1)}} &&& Y_0 \times Y_0.
}
\]
By the part of the lemma that we have already proved, the map $F_0|_{s(C_1)} \times F_0|_{t(C_1)}$ is proper. The map $s \times t$ is proper by the definition of \EPG{}. Its restriction to $C_1$ is proper because $C_1$ is closed by definition. So, the commutativity of the diagram implies that $F_1|_{C_1}$ is proper.
\end{proof}

\begin{lm}\label{lm:pcc}
Let $F : \X \to \Y$ be a morphism in \EPG{} and let $C_i \subset X_i$ be clean subsets. If $F$ is proper, then $F_i(C_i) \subset Y_i$ is clean.
\end{lm}
\begin{proof}
First we prove that $F_0(C_0)\subset Y_0$ is clean. Indeed, by Lemma~\ref{lm:cpp} it suffices to prove the map $\pi|_{F_0(C_0)} : F_0(C_0) \to |\Y|$ is proper. Since $|\Y|$ is Hausdorff, a compact subset $K \subset |\Y|$ is closed. So, it suffices to prove that the closed set $\left(\pi|_{F_0(C_0)}\right)^{-1}(K)\subset F_0(C_0)$ is contained in a compact set. This follows from the properness of $|F|,$ Lemma~\ref{lm:cpp}, and the commutativity of the following diagram:
\[
\xymatrix{
C_0 \ar[r]^(.42){F_0|_{C_0}} \ar[d]^{\pi|_{C_0}} & F_0(C_0) \ar[d]^{\pi|_{F_0(C_0)}} \\
|\X| \ar[r]^{|F|} & |\Y|.
}
\]
To prove that $F_1(C_1)\subset Y_1$ is clean, we need to show that $F_1(C_1)$ is closed and $s(F_1(C_1))$ and $t(F_1(C_1))$ are clean. Indeed, since $F_1|_{C_1}$ is proper by Lemma~\ref{lm:proper}, it follows from Lemma~\ref{lm:seqprop} that $F_1(C_1)$ is closed. Furthermore,
$s(F_1(C_1))= F_0(s(C_1)),$ which is clean by what we have already proved, and the same argument works for $t(F_1(C_1)).$
\end{proof}

\begin{proof}[Proof of Lemma~\ref{lm:pullpush}]
To prove part~(\ref{item:pull}), observe that for $\alpha \in A^*(Y_0)$ we have
\[
s^*F_0^*\alpha - t^*F_0^*\alpha = F_1^*(s^*\alpha - t^*\alpha) = 0.
\]
We turn to the proof of part~(\ref{item:push}). By Lemma~\ref{lm:proper} the maps $F_i : X_i \to Y_i$ are proper when restricted to clean subsets, so they give rise to push-forward maps on differential forms with clean support. Lemmas~\ref{lm:pcc} and~\ref{lm:clsc} imply that if $\alpha \in \Acl^*(X_i),$ then the push-forward $(F_i)_*\alpha$ has clean support and so belongs to $\Acl^*(Y_i).$ Finally, recalling the definition~\eqref{eq:relor} of a relative orientation for $F$, by Proposition~\ref{prop:proppp}~\eqref{prop:pushcomp} for $\alpha \in \Acl^*(X_1)$ we have
\[
(F_0)_*(s_*\alpha - t_*\alpha) = s_*(F_1)_*\alpha - t_*(F_1)_*\alpha.
\]
So, $(F_0)_*$ carries $\Im(s_* - t_*)$ to $\Im(s_*-t_*)$ as desired.
\end{proof}

For the next two lemmas, let $F,G: \X \to \Y$ be morphisms in $\EPG{}$ and let $\alpha :F \Rightarrow G$ be a $2$-morphism.
\begin{lm}\label{lm:2morprop}
Suppose $F$ is proper. If $C \subset X_0$ is clean, then $\alpha|_C : C \to Y_1$ is proper.
\end{lm}
\begin{proof}
Let $K \subset Y_1$ be compact. We show that $\alpha^{-1}(K) \cap C$ is compact. Indeed, $s(K)$ is compact as the continuous image of a compact set, so Lemma~\ref{lm:proper} implies that $F_0^{-1}(s(K))\cap C$ is compact. Since $s\circ \alpha = F_0,$ it follows that $\alpha^{-1}(K) \subset F_0^{-1}(s(K)).$ Thus, keeping in mind Lemma~\ref{lm:cc}, we see that $\alpha^{-1}(K)\cap C$ is compact as a closed subset of a compact set.
\end{proof}
\begin{lm}\label{lm:2morcl}
Suppose $F$ and $G$ are proper. If $C \subset X_0$ is clean, then $\alpha(C) \subset Y_1$ is clean.
\end{lm}
\begin{proof}
Since $C$ is closed by definition, and $\alpha|_C$ is proper by Lemma~\ref{lm:2morprop}, it follows by Lemma~\ref{lm:seqprop} that $\alpha(C)$ is closed. Furthermore, we have $s(\alpha(C)) = F_0(C)$ and $t(\alpha(C)) = G_0(C).$ So the claim follows from Lemma~\ref{lm:pcc}.
\end{proof}

\begin{proof}[Proof of Lemma~\ref{lm:2mor}]
We have $\alpha : X_0 \to Y_1$ with $s\circ \alpha = F_0$ and $t \circ \alpha = G_0.$ So, if $\eta \in A^*(\Y),$ then
\[
F_0^*\eta - G_0^*\eta = (s \circ \alpha)^*\eta - (t\circ \alpha)^*\eta = \alpha^*(s^*\eta - t^*\eta) = 0.
\]
On the other hand, by Lemma~\ref{lm:2morprop}, the map $\alpha : X_0 \to Y_1$ is proper when restricted to a clean subset, so it gives rise to a push-forward map on differential forms with clean support. Lemmas~\ref{lm:2morcl} and~\ref{lm:clsc} imply that if $\xi \in \Acl^*(X_0),$ then the push-forward $\alpha_*\xi$ has clean support and so belongs to $\Acl^*(Y_1).$
Finally, recalling the definition~\eqref{eq:arelor} of a relative orientation for $\alpha$, by Proposition~\ref{prop:proppp}~\eqref{prop:pushcomp} if $\xi \in \Acl^*(X_0),$ then
\[
(F_0)_*\xi - (G_0)_*\xi = (s\circ \alpha)_*\xi - (t\circ \alpha)_*\xi = (s_* - t_*)\alpha_*\xi.
\]
So,
\[
F_*[\xi] - G_*[\xi] = [(s_*-t_*)\alpha_*\xi] = 0.
\]
\end{proof}

\subsubsection{Clean subsets and groupoid composition}
\begin{lm}\label{lm:csfp}
Let $Y,Z,W,$ be topological spaces with $W$ Hausdorff, and let $p_1: Y \times_W Z \to Y$ and $p_2 : Y \times_W Z \to Z$ denote the projections. If $A \subset Y \times_W Z$ is closed and there exist compact subsets $B_1 \subset Y$ and $B_2 \subset Z$ such that $p_i(A) \subset B_i$ for $i = 1,2,$ then $A$ is compact.
\end{lm}
\begin{proof}
Since $W$ is Hausdorff, $Y \times_W Z \subset Y \times Z$ is closed. So, $A$ is closed in $Y \times Z.$ Since $A$ is contained in the compact set $B_1 \times B_2$ and $A$ is closed, it follows that $A$ is compact.
\end{proof}
The following lemma is familiar from algebraic geometry. We provide a proof in the purely topological context for the reader's convenience.
\begin{lm}\label{lm:bcp}
Consider the following fiber square.
\[
\xymatrix{
Y \times_W Z \ar[r]^(.6){p_2}\ar[d]^{p_1} & Z \ar[d]^{g} \\
Y \ar[r]^f & W
}
\]
Suppose $Z$ and $W$ are Hausdorff. If $f$ is proper, then $p_2$ is proper.
\end{lm}
\begin{proof}
Let $K \subset Z$ be compact. We prove that $p_2^{-1}(K)$ is compact. By Lemma~\ref{lm:csfp}, it suffices to show that $p_2^{-1}(K)$ is closed and $p_i(p_2^{-1}(K))$ is compact for $i = 1,2.$ Indeed, since $Z$ is Hausdorff, $K$ is closed, and hence $p_2^{-1}(K)$ is closed. Clearly $p_2(p_2^{-1}(K)) = K$ is compact. Finally, $p_1(p_2^{-1}(K)) = f^{-1}(g(K))$ is compact since $f$ is proper.
\end{proof}

\begin{lm}\label{lm:preJwd}
Consider the fiber square
\[
\xymatrix{
X_1 \times_{X_0} X_1 \ar[r]^(.6){p_2}\ar[d]^{p_1} & X_1 \ar[d]^t \\
X_1 \ar[r]^s & X_0
}
\]
as well as the composition map $m : X_1 \times_{X_0} X_1 \to X_1.$ Let $C \subset X_1$ be clean. Then,
\begin{enumerate}
\item\label{item:p2|}
$p_2|_{p_1^{-1}(C)}$ is proper;
\item\label{item:m|}
$m|_{p_1^{-1}(C)}$ is proper;
\item\label{item:s|m}
$s|_{m(p_1^{-1}(C))}$ is proper.
\end{enumerate}
\end{lm}
\begin{proof}
Keeping in mind Lemma~\ref{lm:sgc}, part~\eqref{item:p2|} is a special instance of Lemma~\ref{lm:bcp}. To prove part~\eqref{item:m|}, we show that for any compact $K \subset X_1,$ the preimage
\[
(m|_{p_1^{-1}(C)})^{-1}(K) = m^{-1}(K) \cap p_1^{-1}(C)
\]
is compact. Consider the following commutative diagram.
\[
\xymatrix{
& X_1 \times_{X_0} X_1 \ar[ld]_{p_1}\ar[rrrr]^m\ar[rdd]^(.4){p_2}
\save[]-<2cm,-.04cm>*{p_1^{-1}(C)\subset} \restore
&&&& X_1 \ar[lld]_t\ar[rrdd]^s
\save[]+<.75cm,.05cm>*{\supset K} \restore
 \\
X_1 \ar[rrr]|(.56)\hole^(.42)t \ar[rdd]^(.4)s
\save[]-<.75cm,-.03cm>*{C \subset} \restore
&&& X_0
\\
&& X_1 \ar[ld]_(.6)t \ar[rrrrr]^s
\save[]-<1.5cm,-.05cm>*{t^{-1}(s(C))\subset } \restore
&&&&& X_0 \\
&X_0
}
\]
By Lemma~\ref{lm:csfp}, it suffices to show $m^{-1}(K)\cap p_1^{-1}(C)$ is closed and $p_i(m^{-1}(K)\cap p_1^{-1}(C))$ is contained in a compact set for $i = 1,2.$
On the one hand, the restriction $t|_C$ is proper by Lemma~\ref{lm:sgc}. So, $(t|_C)^{-1}(t(K))$ is a compact set, and it contains $p_1(m^{-1}(K)\cap p_1^{-1}(C))$ by the commutativity of the diagram.
On the other hand, since $C$ is clean, so is $s(C)$, and consequently $s|_{t^{-1}(s(C))}$ is proper. Observe also that $t^{-1}(s(C)) = p_2(p_1^{-1}(C)).$ So, $(s|_{t^{-1}(s(C))})^{-1}(s(K))$ is a compact set and it contains $p_2(m^{-1}(K)\cap p_1^{-1}(C))$ by the commutativity of the diagram. Finally, since $C$ is closed by definition, it follows that $m^{-1}(K)\cap p_1^{-1}(C)$ is closed.

We now prove part~\eqref{item:s|m}. Since $m|_{p_1^{-1}(C)}$ is proper, $m(p_1^{-1}(C))$ is closed. Furthermore, since $C$ is clean, so is $t(C),$ and consequently $s|_{t^{-1}(t(C))}$ is proper. By commutativity of the diagram, $m(p_1^{-1}(C)) \subset t^{-1}(t(C)),$ so $s|_{m(p_1^{-1}(C))}$ is proper as the restriction of a proper map to a closed subset.
\end{proof}

\subsection{Comparing differential forms and cleanly supported differential forms} \label{ssec:compare}
Recall that the push-forward of differential forms by any of the structure maps of an \'etale proper groupoid with corners, which are by definition local diffeomorphisms, is always taken with respect to the canonical relative orientation as defined in Section~\ref{ssec:orientations}.

\begin{proof}[Proof of Lemma~\ref{lm:Jwd}]
To prove part~\eqref{item:tst}, consider the following diagram.
\[
\xymatrix{
& X_1 \ar[dl]_s \ar@<3pt>[dd]^i \ar[dr]^t \\
X_0 && X_0 \\
&X_1 \ar[ul]^t \ar@<3pt>[uu]^i \ar[ur]_s
}
\]
By Lemma~\ref{lm:pp}, since $i \circ i = \Id,$ we have $i^* = i_*.$ Thus, by Proposition~\ref{prop:proppp}~\eqref{prop:pushcomp} for $\alpha \in \Acl^*(X_0)$ we have
\[
t_*s^*\alpha = t_* i^* t^*\alpha = t_* i_* t^*\alpha = s_* t^*\alpha.
\]
Next, we prove part~\eqref{item:kernelJ}. Suppose $\eta \in \Acl^*(X_1)$ and $\alpha = s_*\eta - t_*\eta.$ The following diagram commutes.
\[
\xymatrix{
&X_1 \times_{X_0} X_1 \ar[dl]_{p_1}\ar[dr]^{p_2}\ar[rr]^m && X_1 \ar[dd]^s \\
X_1\ar[dr]^s && X_1 \ar[dl]_t \ar[dr]^s \\
& X_0 && X_0
}
\]
We equip $p_2$ with the canonical relative orientation of a local diffeomorphism, which coincides with pull-back orientation $s^*o^t_c$ by Lemma~\ref{lm:pb&c}.
So, keeping in mind Lemma~\ref{lm:preJwd}, Proposition~\ref{prop:proppp} implies that
\begin{equation}\label{eq:sts}
s_*t^*s_* \eta = s_*m_*p_1^*\eta.
\end{equation}
On the other hand, consider the local diffeomorphism
\[
q : X_1 \times_{X_0} X_1 \to X_1 \times_{X_0} X_1
\]
given by $q(x,y) = (i(x),m(x,y)).$ Equip $q$ with the canonical relative orientation. Observe that $q\circ q = \Id,$ so $q$ is in  fact a diffeomorphism and Lemma~\ref{lm:pp} gives $q^* = q_*.$ The following diagram commutes.
\[
\xymatrix{
X_1\ar@<3pt>[dd]^i && X_1 \times_{X_0} X_1 \ar@<3pt>[dd]^{q}\ar[ll]_{p_1} \ar[rr]^(.6)m && X_1 \ar[dr]^s \\
 &&&&& X_0 \\
X_1 \ar@<3pt>[uu]^i &&X_1 \times_{X_0} X_1 \ar@<3pt>[uu]^q \ar[ll]_{p_1}\ar[rr]^(.6)m && X_1\ar[ur]_s
}
\]
So, Proposition~\ref{prop:proppp}~\eqref{prop:pushcomp} gives
\begin{equation}\label{eq:smpi}
s_*m_*p_1^*i_*\eta = s_*m_*p_1^*i^*\eta = s_*m_*q^*p_1^*\eta = s_* m_* q_* p_1^*\eta = s_* m_*p_1^*\eta.
\end{equation}
Combining equations~\eqref{eq:sts} and~\eqref{eq:smpi}, we conclude that
\[
s_*t^* \alpha = s_*t^*s_* \eta - s_*t^*t_* \eta = s_* t^* s_* \eta - s_* t^* s_* i_*\eta = s_*m_*p_1^*\eta - s_* m_* p_1^*i_*\eta = 0,
\]
as desired.

The proof of part~\eqref{item:imageJ} is similar. Equipping $p_1$ with the canonical relative orientation of a local diffeomorphism, we derive from the above diagrams that
\[
s^*t_*s^*\alpha = (p_1)_* m^*s^*\alpha,  \qquad (p_1)_* m^*s^*\alpha = i^*(p_1)_* m^*s^*\alpha.
\]
So, it follows from part~\eqref{item:tst} that
\begin{multline*}
s^*s_*t^*\alpha - t^*s_*t^*\alpha = s^*t_*s^*\alpha - t^*t_*s^*\alpha =\\
 =s^*t_*s^*\alpha - i^*s^*t_*s^*\alpha = (p_1)_* m^*s^*\alpha - i^*(p_1)_* m^*s^*\alpha = 0.
\end{multline*}
\end{proof}

\begin{lm}\label{lm:FJF = J}
Let $F : \X \to \Y$ be a proper local diffeomorphism in \EPG{} that is fully faithful as a functor of the underlying groupoids. Equipping $F$ with the canonical relative orientation, the following diagram commutes.
\[
\xymatrix{
A^*(\X) & A^*(\Y) \ar[l]_(.45){F^*} \\
\Acl^*(\X)\ar[u]^J\ar[r]^{F_*} & \Acl^*(\Y) \ar[u]^{J}
}
\]
\end{lm}
\begin{proof}
By definition, $F$ is fully faithful if and only if the following diagram is Cartesian.
\[
\xymatrix{
X_1 \ar[r]^{F_1}\ar[d]^{t\times s} & Y_1 \ar[d]^{t\times s} \\
X_0 \times X_0 \ar[r]^{F_0 \times F_0} & Y_0 \times Y_0
}
\]
In other words, we have a canonical diffeomorphism
\[
X_1 \simeq (X_0 \times X_0)\times_{Y_0\times Y_0} Y_1.
\]
Furthermore, we have a canonical diffeomorphism
\[
(X_0 \times X_0)\times_{Y_0\times Y_0} Y_1 \simeq (X_0 \times_{Y_0} Y_1) \times_{Y_0} X_0.
\]
Thus, we obtain the following commutative diagram, in which both squares are Cartesian.
\[
\xymatrix{
& X_1 \ar@/_1pc/[ddl]_t \ar[d]^\wr \ar[drr]^s \\
& (X_0 \times_{Y_0} Y_1) \times_{Y_0} X_0 \ar[d]^{q_1} \ar[rr]^(.55){q_2} && X_0 \ar[d]^{F_0} \\
X_0 \ar[d]^{F_0}& X_0 \times_{Y_0} Y_1 \ar[l]^{p_1}\ar[d]_{p_2} \ar[rr]^{s\circ p_2} && Y_0 \\
Y_0 & Y_1 \ar[l]^t \ar[urr]^s
}
\]
All maps in the diagram are local diffeomorphisms, and we equip them with their canonical relative orientations. So, for $\alpha \in \Acl^*(X_0)$ Lemma~\ref{lm:pb&c}, Proposition~\ref{prop:proppp}, and Lemma~\ref{lm:pp}, give
\begin{equation*}
F_0^*t_*s^*(F_0)_*\alpha = (p_1)_* p_2^*s^*(F_0)_*\alpha = (p_1)_* (s \circ p_2)^*(F_0)_*\alpha
= (p_1)_*(q_1)_*q_2^*\alpha = t_* s^* \alpha,
\end{equation*}
as desired.
\end{proof}

\begin{lm}\label{lm:Frho}
Suppose $F : \X \to \Y$ is a refinement.
If $\rho$ is a partition of unity for $\X$, then $F_*\rho$ is a partition of unity for $\Y.$
\end{lm}
\begin{proof}
The claim is equivalent to $J  F_*\rho = 1.$
It follows from Lemma~\ref{lm:FJF = J} that
\[
F^*  J F_*\rho = J \rho = 1.
\]
So, it suffices to show that if $\alpha \in A^*(\Y)$ satisfies $F^* \alpha = 1,$ then $\alpha = 1.$
Indeed, consider the following fiber product.
\[
\xymatrix{
X_0 \times_{Y_0} Y_1 \ar[r]^(.65){p_2}\ar[d]_{p_1} & Y_1 \ar[d]^t \\
X_0 \ar[r]^{F_0} & Y_0
}
\]
Since $F$ is an equivalence of categories, it is in particular essentially surjective, which means that the map
\[
X_0 \times_{Y_0} Y_1 \overset{s\circ p_2}\lrarr Y_0
\]
is surjective. Since $\alpha \in A^*(\Y),$ we have $s^*\alpha = t^*\alpha,$ so
\[
1 = p_1^* 1 = p_1^* F_0^*\alpha = p_2^* t^*\alpha = p_2^* s^*\alpha = (s\circ p_2)^*\alpha.
\]
Since $s \circ p_2$ is a surjective local diffeomorphism, the claim follows.
\end{proof}
\begin{proof}[Proof of Lemma~\ref{lm:pu}]
It is shown in~\cite[pp. 14-15]{Behrend} that there exists a refinement $F: \X'\to \X$ such that there exists a partition of unity $\rho'$ for $\X'.$ Lemma~\ref{lm:Frho} asserts that $\rho = F_*\rho'$ is a partition of unity for $\X.$
\end{proof}

\begin{proof}[Proof of Lemma~\ref{lm:JKinv}]
First we prove that $J\circ K = \Id.$ Indeed, for $\alpha \in A^*(\X),$ we have
\[
J \circ K(\alpha) = t_*s^*(\rho \alpha) = t_*(s^*\!\rho \,s^*\alpha) = t_*(t^*\alpha\, s^*\!\rho) = \alpha \,t_*s^*\!\rho = \alpha.
\]
Next, we prove that $J$ is injective. More specifically, suppose $\alpha \in \Acl(X_0)$ and $s_*t^*\alpha = 0.$ By Lemma~\ref{lm:pu} choose a partition of unity $\rho \in \Acl^*(X_0).$ We claim that $\alpha = (t_* - s_*)(s^*\!\rho \, t^*\alpha),$ so
$[\alpha] = 0 \in \Acl^*(\X) = \Acl^*(X_0)/\Im(s_*-t_*).$
Indeed,
\[
(t_* - s_*)(s^*\rho \, t^*\alpha) = t_*(t^*\alpha\, s^*\!\rho) - s_*(s^*\!\rho\, t^*\alpha)= \alpha t_*s^*\rho - \rho s_*t^*\alpha = \alpha \cdot 1 -\rho \cdot 0= \alpha.
\]
Since $J \circ K = \Id,$ it follows that $J$ is surjective. So, $J$ is invertible, and hence $K = J^{-1}.$
\end{proof}

\begin{lm}\label{lm:FKF = K}
If $F : \X \to \Y$ is a refinement, then the following diagram commutes.
\[
\xymatrix{
A^*(\X)\ar[d]^K & A^*(\Y) \ar[l]_(.45){F^*}\ar[d]^{K} \\
\Acl^*(\X)\ar[r]^{F_*} & \Acl^*(\Y)
}
\]
\end{lm}
\begin{proof}
By Lemma~\ref{lm:pu} choose a partition of unity $\rho \in \Acl^*(\X).$ By Lemma~\ref{lm:Frho} the push-forward $F_*\rho \in \Acl^*(\Y)$ is a partition of unity. So, for $\alpha \in A^*(\Y)$ Proposition~\ref{prop:proppp}\eqref{prop:pushpull} gives
\[
F_*KF^*\alpha = F_*(\rho F^*\alpha) = (F_*\rho)\alpha = K\alpha.
\]
\end{proof}

\begin{proof}[Proof of Lemma~\ref{lm:JFK}]
Lemma~\ref{lm:JKinv} and Lemma~\ref{lm:FJF = J} give
\[
F^*\circ (J\circ F_*\circ K) = (F^*\circ J\circ F_*)\circ K = J\circ K = \Id.
\]
On the other hand, Lemma~\ref{lm:JKinv} and Lemma~\ref{lm:FKF = K} give
\[
(J\circ F_*\circ K) \circ F^* = J \circ (F_*\circ K\circ F^*) = J \circ K = \Id.
\]
The lemma follows.
\end{proof}

\section{Integration properties}\label{sec:thm1}

\subsection{Fiber products of orbifolds}\label{ssec:fpo}

In a discussion of properties of differential forms, it is useful for us to describe an explicit construction of fiber products in $\EPG$. To that end, we recall the definition of a weak fiber products in a general 2-category $\mC$ as in~\cite[Remark 2.2]{Tommasini}.

Let $F,G,H: \X\to \Y$ be morphisms in $\mC$ such that there are 2-morphisms $\a:F\Rarr G$ and $\beta:G\Rarr H$. Then we denote by
\[
\beta\circ \a: F\Rarr H
\]
the (vertical) composition 2-morphism that attaches to any object $x\in X_0$ the morphism
\[
(\beta\circ \a)_x:=\beta_x\circ \a_x:F(x)\lrarr H(x).
\]

Let $F, H:\X\to \Y$ and $G, L: \Y\to \mZ$ be morphisms such that there are 2-morphisms $\a:F\Rarr H$ and $\beta:G\Rarr L$. Then
we denote by
\[
\beta*\a:G\circ F \Rarr L\circ H
\]
the (horizontal) composition 2-morphism that attaches to any object $x\in X_0$ the morphism
\[
(\beta*\a)_x:= L_1(\a_x)\circ \beta_{F(x)} =\beta_{H(x)}\circ G_1(\a_x) :G(F(x))\lrarr L(H(x)).
\]
In particular,
\[
(\Id_G*\a)_x = G_1(\a_x), \qquad
(\beta *\Id_F)_x = \beta_{F(x)}.
\]

Let $F:\X\to \mZ$ and $G:\Y\to \mZ$ be morphisms in $\mC$.
A \textbf{weak fiber product} of $F,G,$ is a quadruple $(\mP,A_1,A_2,\a)$ where $\mP$ is a 0-cell, $A_1:\mP\to \X$ and $A_2:\mP\to \Y$ are morphisms, and $\a:G\circ A_2\Rarr F\circ A_1$ is an invertible 2-morphism,
\begin{equation}\label{eq:fibprod}
\xymatrix{
\mP \ar[rr]^{A_2}\ar[d]^{A_1} & & \Y \ar[d]^{G}\ar@2{S->S}[dll]_{\a}\\
\X \ar[rr]^{F} & &\mZ ,
}
\end{equation}
such that the two properties below are satisfied.
\begin{enumerate}[I.]
\item\label{it:wup1}
    For any triple $(\D,\; B_1\!:\!\D\to \X,\; B_2\!:\!\D\to \Y)$, if there exists an invertible 2-morphism $\a':G\circ B_2\Rarr F\circ B_1$, then there exists a triple $(U\!:\!\D\to\mP, \; \beta_1\!:\!B_1\Rarr A_1\circ U,\; \beta_2\!:\! B_2\Rarr A_2\circ U)$ such that $\beta_j$ are invertible and
    \[
    (\a*\Id_U) \circ (\Id_G* \beta_2) = (\Id_F*\beta_1) \circ \a' : G\circ B_2 \Rarr F\circ A_1\circ U.
    \]
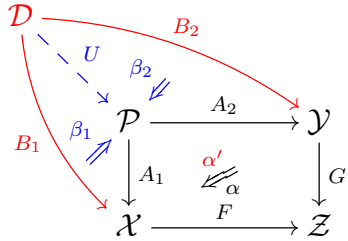
\[
\xymatrix{
\color{red}{\D}\ar@[blue]@{-->}^{\color{blue}{U}}[dr]\ar@[red]@/^10pt/[drrr]^{\color{red}{B_2}} \ar@[red]@/_10pt/[rdd]_{\color{red}{B_1}} & & \ar@[blue]@2{S->}[dl]_(.65){\color{blue}{\beta_2}}&\\
& \mP \ar[rr]^{A_2}\ar[d]^{A_1} & & \Y \ar[d]^{G}\ar@2{S->S}[dll]_{\color{red}{\a'}}^{\a}\\
\ar@[blue]@2{T->}[ur]^(.75){\color{blue}{\beta_1}}& \X \ar[rr]^{F} & &\mZ
}
\]
\item\label{it:wup2}
    For any triple $(\D,\; U\!:\!\D\to \mP,\; U'\!:\!\D\to \mP)$, if there exist invertible 2-morphisms $\gamma_j:A_j\circ U \Rarr A_j\circ U'$ for $j=1,2,$ such that
    \[
    (\a*\Id_{U'})\circ (\Id_{G} * \gamma_2) = (\Id_{F} * \gamma_1)\circ (\a* \Id_{U}): G\circ A_2\circ U\Rarr F\circ A_1\circ U',
    \]
    then there exists a unique invertible $\gamma:U\Rarr U'$ such that
    \[
    \gamma_j=\Id_{A_j}*\gamma,\qquad j=1,2.
    \]
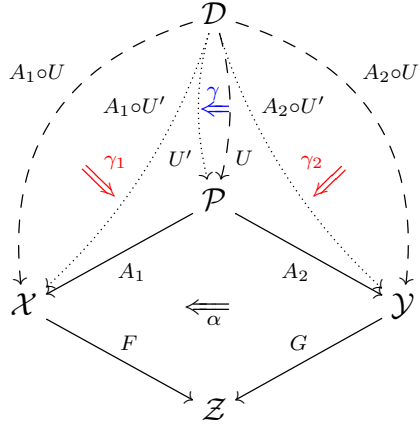
    \[
\xymatrix{
 & & \D \ar@{-->}@/^6pt/[dd]^(.75){U}
 \ar@{..>}@/_6pt/[dd]_(.75){U'} \ar@{-->}@/_30pt/[dddll]_(.4){A_1\circ U} \ar@{..>}@/^10pt/[dddll]_(.3){A_1\circ U'} \ar@{-->}@/^30pt/[dddrr]^(.4){A_2\circ U} \ar@{..>}@/_10pt/[dddrr]^(.3){A_2\circ U'}& & \\
 \ar@2{S->}@[red][dr]^(.75){\color{red}{\gamma_1}} &&
 \ar@2{>T}@[blue][l]_(.01){\color{blue}{\gamma}}
 && \ar@2{S->}@[red][dl]_(.75){\color{red}{\gamma_2}} \\
 & & \mP\ar[drr]_{A_2}\ar[dll]^{A_1} & & \\
\X\ar[drr]^{F} && &\ar@2{S->S}[ll]^{\a}& \Y\ar[dll]_{G}\\
 & & \mZ & &\\
}
\]
\end{enumerate}

\begin{rem}\label{rem:wfpbicat}
In a weak 2-category, the weak fiber product is defined similarly,  except the formulae in properties~\ref{it:wup1}-\ref{it:wup2} involve extra 2-morphisms that compensate for the lack of associativity. See, e.g., Remark 2.2 in~\cite{Tommasini}.
The precise expressions will not be required for our purposes.
\end{rem}

To construct a weak fiber product in $\Orb$, we follow the argument of~\cite[Lemma 26]{Zernik2}.
As a first step, we describe a candidate weak fiber product in $\EPG$. For this, consider morphisms $F:\X\to\mZ$ and $G:\Y\to\mZ$, and define a groupoid $\mP$ by
\begin{gather*}
P_0 = X_0\prescript{}{F_0}\times_{s}^{} Z_1 \prescript{}{t}\times_{G_0}^{} Y_0,\\
P_1 = X_1\prescript{}{s\circ F_1}\times_{s}^{} Z_1 \prescript{}{t}\times_{s\circ G_1}^{} Y_1,
\end{gather*}
with the structure maps defined below.
Let $f=(f_1,f_2,f_3)\in P_1$. Write
\[
x_1=s(f_1),\quad  z_1 = f_2, \quad y_1=s(f_3), \quad x_2=t(f_1), \quad y_2=t(f_3),
\]
and take $z_2$ so that the following diagram commutes.
\[
\xymatrix{
F_0(x_1)\ar[rr]^{z_1}\ar[d]^{F_1(f_1)} & &  G_0(y_1)\ar[d]^{G_1(f_3)}\\
F_0(x_2)\ar@{-->}[rr]^{z_2} & & G_0(y_2)
}
\]
Define the source and the target of $f$ by
\begin{equation}\label{eq:s}
s(f) := (x_1, z_1  , y_1),
\qquad
t(f) := (x_2, z_2 , y_2).
\end{equation}
It is immediate from definition that these indeed belong to $P_0$:
\begin{gather*}
F_0(s(f_1)) = F_0(x_1) = s(z_1),
\qquad
G_0(s(f_3)) = G_0(y_1) = t(z_1),\\
F_0(t(f_1)) = F_0(x_2) = s(z_2),
\qquad
G_0(t(f_3)) = G_0(y_2) = t(z_2).
\end{gather*}
For $p=(x,z,y)\in P_0$, define the identity morphism by
\begin{equation}\label{eq:e}
e(p):=(e(x), z, e(y)).
\end{equation}
To verify that $e(p)\in P_1$, note that $s\circ F_1(e(x))= F_0\circ s(e(x)) = F_0(x) = s(z)$, and similarly $s\circ G_1(e(y))=t(G_1(e(y)))=t(z)$.
For $p,q,r\in P_0$
and $f:p\to q, g:q\to r\in P_1$, write $f=(f_1,f_2,f_3)$, $g=(g_1,g_2,g_3)$, and define inverse and composition by
\begin{equation}\label{eq:i}
i(f):=(i(f_1),\, G_1(f_3)\circ f_2\circ i(F_1(f_1)),\, i(f_3)),
\qquad
m(g,f):=(m(g_1,f_1),\, f_2,\, m(g_3,f_3)).
\end{equation}
To verify the images are indeed in $P_1$, compute
\begin{gather*}
s(F_1(i(f_1)) 
= s(i(F_1(f_1))) = s(G_1(f_3)\circ f_2\circ i(F_1(f_1))),
\\
s(G_1(i(f_3)))
=
t(G_1(f_3)) = t(G_1(f_3)\circ f_2\circ i(F_1(f_1))),\\
s\circ F_1(m(g_1,f_1)) = F_0(s(f_1)) =
s(f_2),\\
s\circ G_1(m(g_3,f_3)) = G_0(s(f_3)) = s(G_1(f_3)) = t(f_2).
\end{gather*}

\begin{lm}\label{lm:fibEPG}
Let $F:\X\to \mZ$ and $G:\Y\to \mZ$ be transverse strongly smooth morphisms in $\EPG$.
Then the weak fiber product $\mP:=\X\prescript{}{F}\times_G\Y$ exists in $\EPG$ with
\begin{gather*}
P_0 = X_0\prescript{}{F_0}\times_{s}^{} Z_1 \prescript{}{t}\times_{G_0}^{} Y_0,\\
P_1 = X_1\prescript{}{s\circ F_1}\times_{s}^{} Z_1 \prescript{}{t}\times_{s\circ G_1}^{} Y_1,
\end{gather*}
and structure maps defined via~\eqref{eq:s},~\eqref{eq:e}, and~\eqref{eq:i}.
\end{lm}
\begin{proof}
By~\cite[Theorem 6.4]{Joyce2} the topological spaces $P_0,P_1,$ are manifolds with corners, and it is immediate that the structure maps make $\mP$ into an \'etal proper groupoid. Recall that in $\EPG$ all 2-morphisms are automatically invertible.

Consider the specified $\mP$ with $A_1,A_2,$ projections on the first and third component, respectively. We construct a 2-morphism $\a: (G\circ A_2)\Rarr (F\circ A_1)$ to complete diagram~\eqref{eq:fibprod} as follows. For any $p=(x,z,y)\in P_0$, define
\[
\a_p:=i(z) :(G\circ A_2)_0(p)=G_0(y)\lrarr (F\circ A_1)_0(p) = F_0(x).
\]
It is immediate from definition that with $\a$ thusly defined, for any $f:p\to q \in P_1$ the diagram
\[
\xymatrix{
G_0(A_{20}(p))\ar[rr]^{\a_p}\ar[d]^{G_1(A_{21}(f))} & & F_0(A_{10}(p)) \ar[d]^{F_1(A_{11}(f))}\\
G_0(A_{20}(q))\ar[rr]^{\a_q} & & F_0(A_{10}(q))
}
\]
commutes. 
Thus, $(\mP, A_1,A_2,\a)$ is a weak fiber product of $F,G$.

If $(\D,B_1,B_2,\a')$ are given as in property~\ref{it:wup1}, define $U:\D\to \mP$ by
\begin{gather*}
U_0(d):=(B_{10}(d),i(\a'_d), B_{20}(d)), \qquad d\in D_0,\\
U_1(\hd):=(B_{11}(\hd), i(\a'_{s(\hd)}), B_{21}(\hd)),\qquad \hd\in D_1.
\end{gather*}
It is easy to see that the image of $U_j$ is indeed in $P_j$ for $j=0,1$ and that $B_j=A_j\circ U$.
Take $\beta_j:=\Id_{B_j}$, $j=0,1$. Then
relation we need to prove becomes $\a*\Id_U = \a'$. To verify it, note that for any $d\in D_0$ the definition of $\a$ gives
\[
(\a * \Id_U)_d = \a_{U(d)}
=
\a_{(B_{10}(d),i(\a'_d),B_{20}(d))}
= i(i(\a'_d)) = \a'_d,
\]
as desired.

To verify property~\ref{it:wup2}, let $(\D, U, U', \gamma_1, \gamma_2)$ be a quintuple that satisfies the assumptions.
Set
\[
\gamma := \gamma_1\times  i(\a * \Id_U)\times \gamma_2.
\]
Thus defined, $\gamma$ satisfies $\gamma_j=\Id_{A_j}*\gamma$ for $j=1,2$. Since for any $d\in D_0$ the following square is Cartesian,
\[
\xymatrix{
Hom_{\mP}(U_0(d),U'_0(d))\ar[r]^{A_{21}}\ar[dd]^{A_{11}} & Hom_{\Y}(A_{20}(U_0(d)),A_{20}(U'_0(d)))\ar[d]^{G_1}\\
& Hom_{\mZ}(G_0(A_{20}(U_0(d),G_0(A_{20}(U'_0(d))))\ar[d]^{\a\circ -\circ i(\a)}\\
Hom_{\X}(A_{10}(U_0(d)),A_{10}(U'_0(d)))\ar[r]^(.45){F_1} & Hom_{\mZ}(F_0(A_{10}(U_0(d))),F_0(A_{10}(U'_0(d)))),
}
\]
the morphism $\gamma_d$ is uniquely determined by its projections $A_{j1}(\gamma_d)=(\Id_{A_j}*\gamma)_d$, $j=1,2$. So, $\gamma$ is unique.
\end{proof}
In the situation of Lemma~\ref{lm:fibEPG}, consider the weak pull-back diagram
\[
\xymatrix{
\X\prescript{}{F}\times^{}_G\Y \ar[rr]^{A_2}\ar[d]^{A_1} & & \Y \ar[d]^{G}\ar@2{S->S}[dll]_{\a}\\
\X \ar[rr]^{F} & &\mZ.
}
\]
Given a relative orientation $o^G = (o^{G_0},o^{G_1})$ of $G,$ we define the \textbf{pull-back relative orientation} $F^* o^G$ of $A_1$ as follows. Consider the following diagram in which all squares are Cartesian.
\[
\xymatrix{
X_0\prescript{}{F_0}\times_{s}^{}Z_1\prescript{}{t}\times_{G_0}^{} Y_0
\ar[r]^(.6){B^1_0}\ar[d]^{C^1_0}\ar@/_30pt/[dd]_{A_{10}} \ar@/^30pt/[rr]^{A_{20}}&
Z_1\prescript{}{t}\times_{G_0}^{} Y_0 \ar[r]^(.6){B^2_0}\ar[d]^{T_0} & Y_0\ar[d]^{G_0} \\
X_0\prescript{}{F_0}\times_{s}^{} Z_1 \ar[r]^(.6){S_0}\ar[d]^{C^2_0} & Z_1 \ar[d]^{s}\ar[r]^{t} & Z_0\\
X_0 \ar[r]^{F_0} & Z_0&
}
\]
Let $o^{C^1_0} := S_0^* t^* o^{G_0}$ and let $o^{A_{10}} := o^{C^2_0}_c \circ o^{C^1_0}.$ Consider also the following diagram in which all squares are Cartesian.
\[
\xymatrix{
X_1\prescript{}{s\circ F_1}\times_{s}^{}Z_1\prescript{}{t}\times_{s\circ G_1}^{} Y_1
\ar[r]^(.6){B^1_1}\ar[d]^{C^1_1}\ar@/_30pt/[dd]_{A_{11}} \ar@/^30pt/[rr]^{A_{21}}&
Z_1\prescript{}{t}\times_{s \circ G_1}^{} Y_1 \ar[r]^(.6){B^2_1}\ar[d]^{T_1} & Y_1\ar[d]^{s \circ G_1} \\
X_1\prescript{}{s\circ F_1}\times_{s}^{} Z_1 \ar[r]^(.6){S_1}\ar[d]^{C^2_1} & Z_1 \ar[d]^{s}\ar[r]^{t} & Z_0\\
X_1 \ar[r]^{s\circ F_1} & Z_0&
}
\]
Let $o^{C^1_1} := S_1^* t^* (o^s_c \circ o^{G_1})$ and let $o^{A_{11}} := o^{C^2_1}_c \circ o^{C^1_1}.$ Then, $F^* o^G := (o^{A_{10}},o^{A_{11}}).$ Similarly, given a relative orientation $o^F$ of $F$ we can define the transpose pull-back orientation $G^* o^F$ of $A_2.$

The following is true by~\cite[Corollary 0.3 and Theorem 0.2]{Tommasini}.
\begin{lm}\label{cl:fibprodorb}
Let $f=F|R:\X\leftarrow \X'\rarr \mZ$, $g=G|Q:\Y\leftarrow \Y'\rarr \mZ$ be transverse strongly smooth morphisms in $\Orb$.
Let $(\widehat{\mP}=\X'\prescript{}{F}\times_{G}^{}\Y', \hat{A}_1,\hat{A}_2,\hat{\a})$ be a weak fiber product in $\EPG$ as described in Lemma~\ref{lm:fibEPG}. Then a weak fiber product $\mP=\X\prescript{}{f}\times_g^{}\Y$ exists in $\Orb$ that is given by
\[
P_0 :=\widehat{P}_0 = X'_0\prescript{}{F_0}\times_s^{} Z_1\prescript{}{t}\times_{G_0}^{} Y'_0,
\quad
P_1 :=\widehat{P}_1 = X'_1\prescript{}{s\circ F_1}\times_s^{} Z_1\prescript{}{t}\times_{s\circ G_1}^{} Y'_1,
\]
with the projections
\[
a_1:=(R\circ \hat{A}_1)|\Id_{\mP}:\mP\lrarr \X,
\qquad
a_2:=(Q\circ \hat{A}_2)|\Id_{\mP}:\mP\lrarr \Y,
\]
and the 2-morphism
$\a:=\Id_{\mP}*\hat{\a}$.
\end{lm}

We can illustrate this construction via a commutative diagram in $\EPG$:
\[
\xymatrix{
{\X\prescript{}{f}\times_{g}^{}\Y}& & {\X\prescript{}{f}\times_{g}^{}\Y} \ar[ll]_{\Id}\ar[rr]^(.65){Q\circ\hat{A}_2} &&
{\Y}\\
{\X\prescript{}{f}\times_{g}^{}\Y}\ar[u]_{\Id}\ar[d]^{R\circ \hat{A}_1} &&
{\color{grayish}{\X'\prescript{}{F}\times_{G}^{}\Y'}} \ar@[grayish][rr]^(0.7){\color{grayish}{\hat{A}_2}} \ar@[grayish][d]_{\color{grayish}{\hat{A}_1}} \ar@[grayish]@{-->}[u]^{\color{grayish}{\Id}}
\ar@[grayish]@{-->}[ll]_{\color{grayish}{\Id}}
& &\Y'\ar[u]_Q\ar[d]^G\\
{\X}&& {\X'}\ar[rr]^{F}\ar[ll]_R && {\mZ}
}
\]

In the situation of Lemma~\ref{cl:fibprodorb}, given a relative orientation $o^g$ of $g$, which is the same as a relative orientation $o^G$ of $G,$ the \textbf{pull-back relative orientation} $f^*o^g$ of $a_1,$ which is the same as an orientation of $R\circ \hat A_1,$ is given by $f^* o^{g} = o_c^R \circ F^* o^G.$ Given a relative orientation $o^f$ of $f,$ the transpose pull-back  orientation $^t\!g^*o^f$ of $a_2$ is defined similarly.

\subsection{Integration properties in EPG}

The objective of this section is to prove integration properties for differential forms on \'etale proper groupoids, particularly the analogue  Theorem~\ref{thm:main}.
Note that the property analogous to Theorem~\ref{thm:main}\ref{nt} is covered by Lemma~\ref{lm:2mor} above.

\subsubsection{Property~\ref{opushcomp}}

The claim on composition of pull-backs is immediate from definition. For push-forward, we detail as follows.

\begin{lm}\label{lm:compepg}
Let $G: \X\to \Y$, $F:\Y\to \cZ,$ be relatively oriented proper submersions in $\EPG$. Then
$F_*\circ G_* = (F\circ G)_*$.
\end{lm}
\begin{proof}
Let $\xi\in A^*(\X)$.
By Lemma~\ref{lm:JKinv} and Proposition~\ref{prop:proppp}\eqref{prop:pushcomp} applied to $F_0,G_0,$ we have
    \[
    F_*G_*\xi = J(F_0)_*KJ(G_0)_*K\xi
    = J(F_0)_*(G_0)_*(K\xi)
    = J(F_0\circ G_0)_*(K\xi)
    = (F\circ G)_*\xi.
    \]
\end{proof}

\subsubsection{Property~\ref{opushpull}}

First we establish two basic results.

\begin{lm}\label{lm:AcMod}
$\Acl^*(\X)$ is a module over $A^*(\X)$.
\end{lm}
\begin{proof}
Let $\xi\in A^*(\X)$ and $\eta\in \Im(s_*-t_*)\subset \Acl^*(X_0)$. We need to show that $\xi\wedge \eta \in \Im(s_*-t_*)$.
Indeed, take $\zeta$ such that $\eta=(s_*-t_*)(\zeta)$. Then
\begin{align*}
\xi\wedge\eta
&=
\xi\wedge s_*\zeta - \xi \wedge t_*\zeta\\
&=
s_*(s^*\xi\wedge \zeta)-t_*(t^*\xi\wedge\zeta),\\
\shortintertext{and since $s^*\xi=t^*\xi$,}
&=
(s_*-t_*)(s^*\xi\wedge\zeta).
\end{align*}
\end{proof}

\begin{lm}\label{lm:Jmod}
For $\zeta\in A^*(\X)$ and $\eta\in \Acl^*(X_0)$, we have
$J[\zeta\wedge \eta]=\zeta\wedge J[\eta]$.
\end{lm}

\begin{proof}
By assumption, $s^*\zeta=t^*\zeta$. Therefore,
\begin{align*}
J[\zeta\wedge\eta]&= t_*s^*(\zeta\wedge\eta)\\
&=
t_*(s^*\zeta\wedge s^*\eta)\\
&=
t_*(t^*\zeta\wedge s^*\eta),\\
\shortintertext{which, by Proposition~\ref{prop:proppp}\eqref{prop:pushpull},}
&=
\zeta\wedge t_*s^*\eta
=\zeta \wedge J[\eta].
\end{align*}
\end{proof}

We are now ready to prove the integration property.

\begin{lm}\label{lm:modepg}
Let $F:\X\to \Y$ be a relatively oriented proper submersion in $\EPG$ and let $\xi\in A^*(\Y),$ $\eta\in A^*(\X)$. Then
		$
		F_*(F^*\xi\wedge\eta)=\xi\wedge F_*\eta.
		$
\end{lm}

\begin{proof}
Compute
\begin{align*}
F_*(F^*\xi\wedge\eta)
&=
JF_*K(F_0^*\xi\wedge\eta)\\
&=
JF_*[F_0^*\xi\wedge \rho\eta]\\
&=
J[(F_0)_*(F_0^*\xi \wedge \rho \eta)]\\
\shortintertext{which, by Proposition~\ref{prop:proppp}\eqref{prop:pushpull} applied to $F_0$,}
&=
J[\xi\wedge (F_0)_*(\rho\eta)],\\
\shortintertext{and by Lemma~\ref{lm:Jmod},}
&=
\xi \wedge J[(F_0)_*(\rho\eta)]
=
\xi \wedge F_*\eta.
\end{align*}
\end{proof}

\subsubsection{Property~\ref{opushfiberprod}}

We start with a series of three background lemmas. They lead up to Lemma~\ref{lm:partonfib}, which establishes a property of partitions of unity, which in turn is useful for describing the push-forward of differential forms on fiber products.

\begin{lm}\label{lm:top1}
Let
$f:X\to W$, $g:Y\to W$, and $s,t:Z\to W$ be continuous maps between topological spaces
such that $s\times t:Z\to W\times W$ is proper.
Let
$X\prescript{}{f}\times_s^{} Z\prescript{}{t}\times_g^{} Y$
be the fiber product in \textbf{Top}.
Let $A\subset X$ and $B\subset Y$ be compact subsets. Then $A\prescript{}{f}\times_s^{} Z \prescript{}{t}\times_g^{} B$ is compact.
\end{lm}
\begin{proof}
First, note that $A\prescript{}{f}\times_s^{} Z \prescript{}{t}\times_g^{} B$ is a closed subset of the direct product $X\times Z\times Y$.
Next,
\begin{align*}
A\prescript{}{f}\times_s Z \prescript{}{t}\times_g^{} B
&=
\left\{
(a,z,b)\in A\times Z\times B | f(a)=s(z),t(z)=g(b)
\right\} \\
&\subset
\left\{
(a,z,b)\in A\times Z\times B | s(z)\in f(A), t(z)\in g(B)
\right\} \\
&=
A\times (s\times t)^{-1}(f(A)\times g(B))\times B.
\end{align*}
By assumption, $A$ and $B$ are compact. By continuity, $f(A)\times g(B)$ is compact, and by properness, $(s\times t)^{-1}(f(A)\times g(B))$ is compact. Thus, the last line above is a compact subset of $X\times Z\times Y$. So, $A\prescript{}{f}\times_s^{} Z \prescript{}{t}\times_g^{} B$ is compact as a closed subset of a compact set.
\end{proof}

\begin{lm}\label{lm:fibprodclean}
Let $F:\X\to \mZ$, $G:\Y\to \mZ$, be transverse strongly smooth morphisms in $\EPG$ and let $\mP$ be the weak fiber product given by Lemma~\ref{lm:fibEPG}.
If $A\subset X_0$ and $B\subset Y_0$ are clean, then $A\prescript{}{F_0}\times_{s}^{} Z_1 \prescript{}{t}\times_{G_0}^{} B\subset P_0$ is clean.
\end{lm}

\begin{proof}
Write for short 
$C:=A\prescript{}{F_0}\times_s^{} Z_1 \prescript{}{t}\times_{G_0}^{} B\subset P_0$.
Denote by $s_D,t_D,$ the source and target maps of $\D=\X,\Y,\mP$. Denote by $\pi_j$ the projection from $\mP$ to the $j$th component for $j=1,2,3$, with a second index to indicate the projection of objects of $\mP$ vs. morphisms (e.g., $\pi_{10}:P_0\to X_0$ and $\pi_{11}:P_1\to X_1$).
Let $K\subset P_0$ be a compact subset. We show that $t_P^{-1}(K)\cap s_P^{-1}(C)$ is compact.

First, by Lemma~\ref{lm:cc},
$A$ and $B$ are closed. Thus, $C$
and therefore $t_P^{-1}(K)\cap s_P^{-1}(C)$ is closed as well.
Next,
\begin{align*}
t_P^{-1}(K)\cap &s_P^{-1}(C)
\subset\\
&\subset
\big(\pi_{11}(t_P^{-1}(K))\prescript{}{s\circ F_1}\times_s^{} Z_1 \prescript{}{t}\times_{s\circ G_1}^{} \pi_{31}(t_P^{-1}(K))\big)
\cap
\big(s_X^{-1}(A)\prescript{}{s\circ F_1}\times_s^{} Z_1 \prescript{}{t}\times_{s\circ G_1}^{} s_Y^{-1}(B)\big)\\
&=
\big(\pi_{11}(t_P^{-1}(K))\cap s_X^{-1}(A)\big)
\prescript{}{s\circ F_1}\times_s^{} Z_1 \prescript{}{t}\times_{s\circ G_1}^{}
\big(\pi_{31}(t_P^{-1}(K))\cap s_Y^{-1}(B)\big)\\
&\subset
\big(t_X^{-1}(\pi_{10}(K))\cap s_X^{-1}(A)\big)
\prescript{}{s\circ F_1}\times_s^{} Z_1 \prescript{}{t}\times_{s\circ G_1}^{}
\big(t_Y^{-1}(\pi_{30}(K))\cap s_Y^{-1}(B)\big).
\end{align*}
By continuity, $\pi_1(K)$ and $\pi_3(K)$ are compact. From cleanness of $A$ and $B$ and by Lemma~\ref{lm:top1}, it follows that the last expression above is compact. In total, $t_P^{-1}(K)\cap s_P^{-1}(C)$ is compact.
\end{proof}

Let $X,Y,Z,W,V,$ be manifolds with corners and let
$h:Z\to V$, $k:W\to V$, $f:X\to Z$, $g:Y\to W$, be strongly smooth maps with $k\pitchfork h$ and $f,g,$ proper submersions. Let $o^f,o^g,$ be relative orientations for $f,g,$ respectively.
Denote by
\[
f\times g : X\prescript{}{h\circ f}\times_{k\circ g}^{} Y \to Z\prescript{}{h}\times_k^{} W
\]
the induced map,
and by
\begin{gather*}
\pi_X^Y:X \prescript{}{h\circ f}\times_{k\circ g}^{} Y \to X, \quad \pi_Y^X:X \prescript{}{h\circ f}\times_{k\circ g}^{} Y\to Y,\\
\pi_Z^W:Z \prescript{}{h}\times_{k}^{} W \to Z, \quad \pi_W^Z:Z \prescript{}{h}\times_{k}^{} W \to W,
\end{gather*}
the projections. Equip $f\times g$ with the induced relative orientation defined as follows. Note that
\[
\Id_X\times g: X\prescript{}{h\circ f}\times_{k\circ g}^{} Y \lrarr X\prescript{}{h\circ f}\times_{k}^{} W,
\qquad
f\times \Id_W: X\prescript{}{h\circ f}\times_{k}^{} W \lrarr Z\prescript{}{h}\times_{k}^{} W,
\]
satisfy
\[
f\times g = (f\times \Id_W)\circ (\Id_X\times g).
\]
Consider the following pull-back diagrams
\[
\xymatrix{
X\prescript{}{h\circ f}\times_{k}^{} Y \simeq X\prescript{}{f}\times_{\pi_Z^Y} (Z\prescript{}{h}\times_{k}^{} Y) \ar[r]^(.73){f\times \Id_Y}\ar[d]^{\pi^Y_X}& Z\prescript{}{h}\times_{k}^{} Y\ar[d]^{\pi_Z^Y},\\
X\ar[r]^{f} & Z
}
\qquad
\xymatrix{
X\prescript{}{h}\times_{k\circ g}^{} Y \simeq (X\prescript{}{h}\times_{k}^{} W)\prescript{}{\pi^X_W}\times_{g}^{} Y \ar[r]^(.8){\pi^X_Y}\ar[d]^{\Id_X\times g}& Y\ar[d]^{g}\\
X\prescript{}{h}\times_{k}^{} W\ar[r]^{\pi^X_W} & W.
}
\]
Equip $f\times \Id_Y$ with the transpose pull-back relative orientation $^t(\pi^Y_Z)^*o^f$ and equip $\Id_X \times g$ with the pull-back orientation $(\pi_W^X)^* o^g.$ The induced relative orientation of $f\times g$ is defined to be the composition $^t(\pi^Y_Z)^*o^f \circ (\pi_W^X)^* o^g.$
\begin{lm}\label{lm:pullfibprod}
For all $\a\in A^*(X)$ and $\beta\in A^*(Y)$,
\[
(f\times g)_*((\pi_X^Y)^*\a \wedge (\pi_Y^X)^*\beta)
=
(-1)^{\star}
(\pi_Z^W)^*f_*\a\wedge (\pi_W^Z)^*g_*\beta,
\]
with  $\star = \rdim f \cdot (|\beta| + \rdim k \circ g) .$
\end{lm}
\begin{proof}
By Proposition~\ref{prop:proppp}\eqref{prop:pushcomp}, it is enough to prove the claim for the case when one of the maps is the identity, since then
\begin{multline*}
(f\times g)_*((\pi_X^Y)^*\a \wedge (\pi_Y^X)^*\beta)
=
(f\times \Id_W)_*\circ (\Id_X\times g)_*((\pi_X^Y)^*\a \wedge (\pi_Y^X)^*\beta)
=\\
=
(f\times \Id_W)_*((\pi_X^Y)^*\a\wedge (\pi_W^Z)^*g_*\beta)
=
(-1)^{\rdim f \cdot (|\beta| + \rdim g+ \rdim k )}(\pi_Z^W)^*f_*\a\wedge (\pi_W^Z)^*g_*\beta.
\end{multline*}
Assume first
that $W=Y$ and $g=\Id_Y$.
Consider the following diagrams
\[
\xymatrix{
X\prescript{}{h\circ f}\times_{k}^{} Y \simeq X\prescript{}{f}\times_{\pi_Z^Y} (Z\prescript{}{h}\times_{k}^{} Y) \ar[r]^(.73){f\times \Id_Y}\ar[d]^{\pi^Y_X}& Z\prescript{}{h}\times_{k}^{} Y\ar[d]^{\pi_Z^Y}\\
X\ar[r]^{f} & Z ,
}
\qquad
\xymatrix{
X \prescript{}{h\circ f}\times_{k}^{} Y \ar[dr]^{\pi_Y^X}\ar[rr]^{f\times \Id_Y} & & Z\prescript{}{h}\times_{k}^{} Y \ar[dl]_{\pi_Y^Z}\\
& Y & .
}
\]
By Lemma~\ref{lm:flip} for the first diagram and commutativity of the second,
\[
(\pi_Z^Y)^*f_*\a=
(-1)^{\rdim f\cdot \rdim \pi^Y_Z}
(f\times \Id_Y)_*(\pi_X^Y)^*\a,
\qquad
(\pi_Y^X)^*\beta = (f\times \Id_Y)^*(\pi_Y^Z)^*\beta.
\]
By Proposition~\ref{prop:proppp}\eqref{prop:pushpull},
\begin{align*}
(f\times g)_*((\pi_X^Y)^*\a \wedge (\pi_Y^X)^*\beta)
&=
(f\times \Id_Y)_*((\pi_X^Y)^*\a \wedge (f\times \Id_Y)^*(\pi_Y^Z)^*\beta)\\
&=
(-1)^{|\a||\beta|}
(f\times \Id_Y)_*((f\times \Id_Y)^*(\pi_Y^Z)^*\beta\wedge (\pi_X^Y)^*\a)\\
&=
(-1)^{|\a||\beta|}
(\pi_Y^Z)^*\beta\wedge (f\times \Id_Y)_*(\pi_X^Y)^*\a\\
&=
(-1)^{|\beta|(|\a|+|(f\times \Id_Y)_*(\pi_X^Y)^*\a|)}
(f\times \Id_Y)_*(\pi_X^Y)^*\a \wedge (\pi_Y^Z)^*\beta\\
&=
(-1)^{|\beta|(|\a|+|f_*\a|)+\rdim f\cdot \rdim \pi^Y_Z}
(\pi_Z^Y)^*f_*\a\wedge (\pi_Y^Z)^*(\Id_Y)_*\beta.
\end{align*}
The case $f=\Id_X$ is proved similarly, except no signs arise.
The diagrams
\[
\xymatrix{
X\prescript{}{h}\times_{k\circ g}^{} Y \simeq (X\prescript{}{h}\times_{k}^{} W)\prescript{}{\pi^X_W}\times_{g}^{} Y \ar[r]^(.8){\pi^X_Y}\ar[d]^{\Id_X\times g}& Y\ar[d]^{g}\\
X\prescript{}{h}\times_{k}^{} W\ar[r]^{\pi^X_W} & W ,
}
\qquad
\xymatrix{
X \prescript{}{h}\times_{k\circ g}^{} Y \ar[dr]^{\pi^Y_X}\ar[rr]^{\Id_X\times g} & & X\prescript{}{h}\times_{k}^{} W \ar[dl]_{\pi^W_X}\\
& X & .
}
\]
give
\begin{align*}
(f\times g)_*((\pi_X^Y)^*\a \wedge (\pi_Y^X)^*\beta)
&=
(\Id_X\times g)_*((\Id_X\times g)^*(\pi_X^W)^*\a \wedge (\pi_Y^X)^*\beta)\\
&=
(\pi_X^W)^*\a \wedge (\Id_X\times g)_*(\pi_Y^X)^*\beta\\
&=
(\pi_X^W)^*(\Id_X)_*\a\wedge (\pi_W^X)^*g_*\beta.
\end{align*}
\end{proof}

\begin{lm}\label{lm:partonfib}
Let~\eqref{eq:fibprod} be a weak pull-back diagram in $\EPG$  with $\mP=\X\prescript{}{F}\times_{G}^{}\Y$ as in Lemma~\ref{lm:fibEPG}. Let $\rho_X, \rho_Y,$ be partitions of unity on $\X,\Y,$ respectively. Then $\rho:=A_{10}^*\rho_X\cdot A_{20}^*\rho_Y$ is a partition of unity on $\mP$.
\end{lm}

\begin{proof}
By Lemma~\ref{lm:fibprodclean}, the map $\rho:P_0\to [0,1]$ is cleanly supported.
It is left to verify that $(t_P)_*(s_P)^*\rho=1$.

Write $s_D,t_D,$ for the source and target maps of $\D$ for $\D=\X,\Y,\mZ,\mP$. Let $\pi_1:\mP\to \X$, $\pi_2:\mP\to Z_1$, and $\pi_3:\mP\to \Y$, be the projection maps. We use a second index to indicate the projection of objects of $\mP$ vs. morphisms, e.g., $\pi_{10}:P_0\to X_0$ and $\pi_{11}:P_1\to X_1$. Then
\begin{align*}
(t_P)_*(s_P)^*(\pi_{10}^*\rho_X\wedge \pi_{30}^*\rho_Y)
&=
(t_P)_*(\pi_{11}^*s_X^*\rho_X\wedge \pi_{31}^*s_Y^*\rho_Y),\\
&=
(t_P)_*(\pi_{11}^*(s_X^*\rho_X)\wedge
\pi_{21}^*1\wedge \pi_{31}^*(s_Y^*\rho_Y)),\\
\shortintertext{by Lemma~\ref{lm:pullfibprod} applied repeatedly to $t_P=t_X\times \Id_{Z_1}\times t_Y$,}
&=
\pi_{10}^*((t_X)_*s_X^*\rho_X)\wedge
\pi_{20}^* 1 \wedge
\pi_{30}^*((t_Y)_*s_Y^*\rho_Y)\\
&=
1.
\end{align*}
\end{proof}

We are now ready to prove the integration property.

\begin{lm}\label{lm:ppepg}
	Let~\eqref{eq:fibprod} be a weak pull-back diagram in $\EPG$, where $G$ is a proper submersion with relative orientation $o^G.$ Equip $A_1$ with the pull-back relative orientation $F^*o^G.$ Let $\xi\in A^*(\Y).$ Then
		$
		(A_1)_*(A_2)^*\xi=F^*G_*\xi.
		$
\end{lm}
\begin{proof}
Writing explicitly the weak fiber product as in Lemma~\ref{lm:fibEPG}, we get the following commutative diagram,
\begin{equation}\label{diag:fibprod0}
\xymatrix{
X_0\prescript{}{F_0}\times_{s}^{}Z_1\prescript{}{t}\times_{G_0}^{} Y_0
\ar[r]^(.6){B^1_0}\ar[d]^{C^1_0}\ar@/_30pt/[dd]_{A_{10}} \ar@/^30pt/[rr]^{A_{20}}&
Z_1\prescript{}{t}\times_{G_0}^{} Y_0 \ar[r]^(.6){B^2_0}\ar[d]^{T} & Y_0\ar[d]^{G_0} \\
X_0\prescript{}{F_0}\times_{s}^{} Z_1 \ar[r]^(.6){S}\ar[d]^{C^2_0} & Z_1 \ar[d]^{s}\ar[r]^{t} & Z_0\\
X_0 \ar[r]^{F_0} & Z_0&
}
\end{equation}
where each of the three squares is Cartesian in the category of manifolds with corners, and the right vertical arrow in each is a submersion, while $G_0$ and $T$ are also proper. The relative orientation of $s$ is the canonical one, the relative orientations of $T, C_0^1,$ and $C_0^2,$ are given by pull-back, and $o^{A_{10}} := o^{C_0^2}\circ o^{C_0^1}$.

Let $\rho_X$ and $\rho_Y$ be partitions of unity on $\X,\Y,$ respectively. In view of Lemma~\ref{lm:partonfib}, we take $\rho:=(A_{10})^*\rho_X\cdot (A_{20})^*\rho_Y:P_0\to [0,1]$, a partition of unity on $\mP$.
For $\xi\in A^*(\Y)$, we have
\begin{align*}
(A_{1})_* A_{2}^* \xi
&=
(A_1)_*A_{20}^*\xi \\
&=
J[(A_{10})_*(\rho\cdot A_{20}^*\xi)]\\
&=
J[(A_{10})_*(A_{10}^*\rho_X\cdot A_{20}^*\rho_Y\cdot A_{20}^*\xi)]\\
&=
J[(A_{10})_*(A_{10}^*\rho_X\cdot A_{20}^*(\rho_Y\cdot\xi))]\\
\shortintertext{By Proposition~\ref{prop:proppp}\eqref{prop:pushpull},}
&=
J[\rho_X\cdot (A_{10})_*A_{20}^*(\rho_Y\cdot\xi)]\\
&=
JK((A_{10})_*A_{20}^*(\rho_Y\cdot\xi))\\
&=
(A_{10})_*A_{20}^*(\rho_Y\cdot\xi)\\
&=
(C^2_0)_*(C^1_0)_*(B^1_0)^*(B^2_0)^*(\rho_Y\cdot\xi)\\
\shortintertext{by Proposition~\ref{prop:proppp}\eqref{prop:pushfiberprod},}
&=
(C^2_0)_*S^*T_*(B^2_0)^*(\rho_Y\cdot\xi)\\
&=
F_0^*s_*t^*(G_0)_*(\rho_Y\cdot\xi)\\
\shortintertext{by Lemma~\ref{lm:Jwd}\eqref{item:tst},}
&=
F_0^*t_*s^*(G_0)_*(\rho_Y\cdot\xi)\\
\shortintertext{by Lemma~\ref{lm:Jwd}\eqref{item:kernelJ},}
&=
F^*J(G_0)_*K(\xi)\\
&=
F^*G_*\xi.
\end{align*}
\end{proof}

\subsubsection{Property~\ref{thmstokes}}

For $\X,\Y,$ in $\EPG$ and a strongly smooth $F:\X\to \Y$,
decompose $\d\X$ into vertical and horizontal parts $\d\X=\dv\X\sqcup\dh\X$, where
\[
(\dv\X)_j=\dv X_j, \quad (\dh \X)_j=\dh X_j, \quad j=0,1.
\]
They come with maps
\begin{equation}\label{eq:dvdh}
\iv_{\X}:\dv\X\lrarr \X, \qquad \ih_{\X}:\dh\X\lrarr \X,
\end{equation}
and the structure of an \'etale proper groupoids induced by restricting the structure maps of $\d\X$. Furthermore, the above maps come with relative orientations coming from the compositions $\dv\X\hookrightarrow \d\X \stackrel{i_\X}{\rarr}\X$ and $\dh\X\hookrightarrow \d\X \stackrel{i_\X}{\rarr}\X$.


\begin{lm}\label{lm:partonbd}
If $\rho$ is a partition of unity on $\X$, then $\rho_\d:=i_{\X}^*\rho$ is a partition of unity on $\d\X$.
For a strongly smooth $F:\X\to\Y$,
the map $\rho^v_\d:=(\iv_{\X})^*\rho$ is a partition of unity on $\dv\!\X$.
\end{lm}

\begin{proof}
Denote by $s_{\d},t_{\d},$ the source and target maps of $\d\X$. Denote by $i_0$ and $i_1$ the inclusions $i_1:\d X_1\hookrightarrow X_1$ and $i_0:\d X_0\hookrightarrow X_0$, respectively.
Consider the pull-back diagrams
\[
\xymatrix{
\d X_1\ar[r]^{i_1}\ar[d]^{s_{\d}} & X_1\ar[d]^{s} \\
\d X_0 \ar[r]^{i_0} & X_0,
}\qquad
\xymatrix{
\d X_1\ar[r]^{i_1}\ar[d]^{t_{\d}} & X_1\ar[d]^{t} \\
\d X_0 \ar[r]^{i_0} & X_0.
}
\]
By the commutativity of the diagrams, Proposition~\ref{prop:proppp}\eqref{prop:pushfiberprod}, and the assumption $t_*s^*\rho=1$, we get
\begin{align*}
(t_{\d})_*(s_{\d})^*(i_0^*\rho)
=&
(t_{\d})_*i_1^*s^*\rho\\
=&
i_0^*t_*s^*\rho\\
=&
1.
\end{align*}
A similar argument goes through for $(\iv_\d)^*\rho$.
\end{proof}

As a first step in proving the integration property on $A^*(\X)$, note that its analogue holds on $\Acl^*(\X)$:
\begin{lm}\label{lm:StokesAc}
Let $F:\X\to \Y$ be a strongly smooth relatively oriented proper submersion in $\EPG$, let $m:=\dim \X$, and let $\xi\in \Acl^k(\X)$. Then
$d(F_*\xi)
=
F_*(d\xi)
+(-1)^{m+k}(F|_{\dv\!\X})_*(\xi|_{\dv\!\X})$.
\end{lm}

\begin{proof}
This is immediate from Stokes' theorem, Proposition~\ref{stokes}, for $F_0$.
\end{proof}

To deduce the result for $A^*(\X)$, we note the following two lemmas.
\begin{lm}
The exterior derivative $d$ descends to $\Acl^*(\X)$.
\end{lm}
\begin{proof}
Since $s,t,$ are local diffeomorphisms, their fibers are zero dimensional. By Stokes' theorem, this implies $s_*,t_*,$ commute with $d$, and so $s_*-t_*$ is a chain map. It follows that $d$ descends to the cokernel, $\Acl^*(\X)$.
\end{proof}

\begin{lm}\label{lm:dJK}
The exterior derivative $d$ commutes with $J,K$.
\end{lm}

\begin{proof}
Since $t$ is a local diffeomorphism, and as such has no fiberwise boundary, $d$ commutes with $J=t_*s^*$.
Since $K$ is inverse to $J$, it follows that $d$ commutes with $K$ as well:
\[
d\circ J = J \circ d
\quad
\Lrarr
\quad
d = J \circ d\circ K
\quad
\Lrarr
\quad
K\circ d = d \circ K.
\]
\end{proof}

We are now ready to prove the integration property.
\begin{lm}\label{lm:stokesepg}
Let $F:\X\to \Y$ be a strongly smooth relatively oriented proper submersion in $\EPG$, let $m:=\dim \X$, and let $\xi\in A^k(\X)$. Then
$d(F_*\xi)
=
F_*(d\xi)
+(-1)^{m+k}(F|_{\dv\!\X})_*(\xi|_{\dv\!\X})$.
\end{lm}

\begin{proof}
Applying Lemma~\ref{lm:StokesAc} to $K\xi$ and using Lemma~\ref{lm:dJK}, we get
\begin{align*}
d((F_0)_*(K\xi))
=&
(F_0)_*(d(K\xi))
+(-1)^{m+k}(F_0|_{\dv\!X_0})_*((K\xi)|_{\dv\!\X})\\
=&
(F_0)_*(K (d\xi))
+(-1)^{m+k}(F_0|_{\dv\!X_0})_*((K\xi)|_{\dv\!\X}).
\end{align*}
Applying $J$ on the left and using Lemma~\ref{lm:dJK} again we get
\[
d(J(F_0)_*(K\xi))
=
J(F_0)_*(K (d\xi))
+(-1)^{m+k}J(F_0|_{\dv\!X_0})_*((K\xi)|_{\dv\!\X}),
\]
Note that Lemma~\ref{lm:partonbd} implies
$(K\xi)|_{\dv\!\X}=[\rho|_{\dv\X}\cdot\xi|_{\dv\X}]$, where $\rho$ is a partition of unity on $\X$.
Thus, we get
\[
dF_*(\xi)
=
F_*(d\xi)
+(-1)^{m+k}(F|_{\dv\!\X})_*(\xi|_{\dv\!\X}),
\]
as desired.
\end{proof}

\subsection{Integration properties in Orb}
The objective of this section is to prove Theorem~\ref{thm:main}.
Note that with our convention~\eqref{eq:dfnpush}, Definition~\ref{dfn:push} reads, for $f=F|R$,
\[
f^* =R_*F^*, \qquad f_* = F_*R^*.
\]

\subsubsection{Proof of Theorem~\ref{thm:main}\ref{nt}}
Write $f=F|R$, $g=G|Q$, and let the following diagram represent the 2-morphism $\a:f\Rarr g$.
\[
\xymatrix{
& \X' \ar[dr]_R \ar[drrr]^F\\
\X''' \ar[ur]_{T_1} \ar@{}[rr]|(.56){\Downarrow\, \alpha_1} \ar[dr]^{T_2} & & \X \ar@{}[rr]|(.5){\Downarrow\, \alpha_2} && \Y \\
& \X'' \ar[ur]^Q \ar[urrr]_G
}
\]
By Lemmas~\ref{lm:JFK},~\ref{lm:compepg}, and~\ref{lm:2mor}, we have
\[
g^*
=
Q_*G^*
=
Q_*(T_2)_*T_2^*G^*
=
R_*(T_1)_*T_1^*F^*
=
R_*F^*
=f^*.
\]
Similarly,
\[
g_*
=
G_*Q^*
=
G_* (T_2)_*T_2^*Q^*
=
F_*(T_1)_*T_1^*R^*
=
F_*R^*
=f_*.
\]

\subsubsection{Composition of morphisms}\label{sssec:comp}
For the next two parts of Theorem~\ref{thm:main}, we describe more explicitly how composition works in $\Orb$.
\begin{lm}\label{lm:cartprop}
Let $F:\X\to \mZ$, $G:\Y\to \mZ$, be transverse strongly smooth morphisms in $\EPG$, and let $\mP:=\X\times_{\mZ}\Y$ be their weak fiber product described in Lemma~\ref{lm:fibEPG}.
Let $A_1,A_2,$ be the projections as in diagram~\eqref{eq:fibprod}.
Then
\begin{enumerate}
\item
    If $G_0$ is a local diffeomorphism, then so is $A_{10}$.
\item
    If $G$ is essentially surjective, then so is $A_1$.
\item
    If $G$ is fully faithful, then so is $A_1$.
\item
    If $G$ is a refinement, then so is $A_1$.
\end{enumerate}
\end{lm}
\begin{proof}\leavevmode
\begin{enumerate}
\item
    Consider diagram~\eqref{diag:fibprod0}. Since $G_0$ is a local diffeomorphism, so is $T$ and therefore also $C^1_0$. Since $s$ is a local diffeomorphism, so is $C^2_0$. Therefore, $A_{10}=C_0^2\circ C^1_0$ is a local diffeomorphism.
\item
    Let $x\in X_0$. By assumption on $G$, there exist $z\in Z_1$ and $y\in Y_0$ such that $s(z)=F_0(x)$ and $t(z)=G_0(y)$. Then $p:=(x,z,y)\in P_0$ and $A_{10}(p)=x$. Thus, $A_{10}$ is surjective and in particular $A_1$ is essentially surjective.
\item
    Let $p,q\in P_0$. Denote by
    \[
    \qquad \a_{pq}:Hom_{\mZ}((G_0\circ A_{20})(p),(G_0\circ A_{20})(q)) \lrarr Hom_{\mZ}((F_0\circ A_{10})(p),(F_0\circ A_{10})(q))
    \]
    the map given by
    \[
    \a_{pq}(f):=\a_q\circ f\circ i(\a_p).
    \]
    Then the following square is Cartesian:
    \[
    \xymatrix{
    Hom_{\mP}(p,q)\ar[r]^{A_{21}}\ar[dd]^{A_{11}} & Hom_{\Y}(A_{20}(p),A_{20}(q))\ar[d]^{G_1}\\
    & Hom_{\mZ}(G_0\circ A_{20}(p),G_0\circ A_{20}(q))\ar[d]^{\a_{pq}}\\
    Hom_{\X}(A_{10}(p),A_{10}(q))\ar[r]^(.40){F_1} & Hom_{\mZ}(F_0\circ A_{10}(p), F_0\circ A_{10}(q)).
    }
    \]
    From the assumption on $G$ and the invertibility of $\a$, the right column is a bijection. Therefore, the left column is a bijection.
\item
    This is a combination of the preceding three properties.
\end{enumerate}
\end{proof}

Let $f:\X\overset{Q}{\leftarrow} \X' \overset{G}{\rarr} \Y$ and $g:\Y\overset{R}{\leftarrow} \Y' \overset{F}{\rarr} \mZ$ be morphisms in $\Orb$. Taking the weak fiber product in $\EPG$ of $G$ and $R$, we get the following diagram:
\begin{equation}\label{compdiag}
\xymatrix{
& & \X'\prescript{}{G}\times_{R}^{} \Y'\ar[ld]^{R'}\ar[rd]_{G'} & & \\
 & \X'\ar[ld]_Q\ar[rd]^G\ar@{}[rr]|(.51){\Rightarrow\,}^\a  & & \Y'\ar[ld]_R\ar[rd]^F & \\
\X\ar[rr]^g & & \Y\ar[rr]^f & & \mZ .
}
\end{equation}

By Lemma~\ref{lm:cartprop}, the projection $R'$ is a refinement. Thus, so is $Q\circ R'$, and we can define $f\circ g$ to be $(F\circ G')|(Q\circ R')$.

\begin{rem}
The weak fiber product offers a natural choice of squares required for the definition of composition in~\cite[Section 2.2]{Pronk}, as remarked in~\cite[Lemma 23]{Zernik2}.
\end{rem}

\subsubsection{Proof of Theorem~\ref{thm:main}\ref{opushcomp}}
Write $g=G|Q$, $f=F|R$, and $f\circ g=(F\circ G')|(Q\circ R')$, as in diagram~\eqref{compdiag} above.
By Lemmas~\ref{lm:compepg} and~\ref{lm:ppepg},
\begin{align*}
(f\circ g)^*
=
(Q\circ R')_* (F\circ G')^*
=
Q_* R'_* (G')^* F^*
=
Q_*G^*R_*F^* = g^*f^*.
\end{align*}
Similarly,
\[
(f\circ g)_*
=
(F\circ G')_* (Q\circ R')^*
=
F_* G'_* (R')^* Q^*
=
F_* R^* G_* Q^*
= f_*g_*.
\]

\subsubsection{Proof of Theorem~\ref{thm:main}\ref{opushpull}}
Write $f=F|R$. Then,
\begin{align*}
f_*(f^*\xi\wedge \eta)
&=
F_* R^*(R_* F^*\xi \wedge \eta)\\
&=
F_*(R^*R_* F^*\xi \wedge R^*\eta),\\
\shortintertext{by Lemma~\ref{lm:JFK},}
&=
F_*(F^*\xi  \wedge R^*\eta)\\
\shortintertext{by Lemma~\ref{lm:modepg},}
&=
\xi  \wedge F_*R^*\eta\\
&=
\xi\wedge f_*\eta.
\end{align*}

\subsubsection{Proof of Theorem~\ref{thm:main}\ref{opushfiberprod}}
Write $f:\X\overset{R}{\leftarrow} \X' \overset{F}{\rarr} \mZ$ and $g:\Y\overset{Q}{\leftarrow} \Y' \overset{G}{\rarr} \mZ$.
By Lemma~\ref{cl:fibprodorb}, we may write
$p=(Q\circ \hat{A}_2)|\Id$ and $q=(R\circ \hat{A}_1)|\Id$.
Applying Lemma~\ref{lm:ppepg} to the smaller square in
\[
\xymatrix{
{X'_0\times_{Z_0}Z_1\times_{Z_0} Y'_0}& {X'_0\times_{Z_0}Z_1\times_{Z_0} Y'_0}\ar[l]_{\Id}\ar[rr]^(.65){Q\circ\hat{A}_{20}} &&
{Y_0}\\
{X'_0\times_{Z_0}Z_1\times_{Z_0} Y'_0}\ar[u]_{\Id}\ar[d]^{R\circ \hat{A}_{10}} &
{{X'_0\times_{Z_0}Z_1\times_{Z_0} Y'_0}}\ar[rr]^(0.7){\hat{A}_{20}}\ar[d]_{\hat{A}_{10}} 
& &Y'_0\ar[u]_Q\ar[d]^G\\
{X_0}& {X'_0}\ar[rr]^{F}\ar[l]_R && {Z_0}
}
\]
we get
\[
f^*g_*
=
R_*F^*G_*Q^*
=
R_*(\hat{A}_{10})_*\hat{A}_{20}^*Q^*
=q_*p^*.
\]

\subsubsection{Proof of Theorem~\ref{thm:main}\ref{thmstokes}}
Write $f:\X\overset{R}{\leftarrow} \X' \overset{F}{\rarr} \Y$.
Note that $R$ restricts to a refinement $R|_{\dv\!\X'}: \dv\!\X' \to \dv\!\X$, so $f\big|_{\dv\!\X}=(F\big|_{\dv\!\X'})|(R\big|_{\dv\!\X'})$.
Applying Lemma~\ref{lm:stokesepg} to $R^*\xi$ we get
\begin{align*}
f_*(d \xi)+(-1)^{s+t}\big(f\big|_{\dv\!\X}\big)_*\xi
&=
F_*R^*(d \xi)+(-1)^{s+t}(F\big|_{\dv\!\X'})_* (R\big|_{\dv\!\X'})^*(\xi\big|_{\dv\!\X})\\
&=
F_*(dR^*\xi)+(-1)^{s+t}(F\big|_{\dv\!\X'})_* (R^*\xi)\big|_{\dv\!\X'}\\
&=
d (F_*R^*\xi)\\
&= d(f_*\xi).
\end{align*}

\section{Currents}\label{sec:curr}

Let $\X$ be an object of $\Orb$.
We denote by $\Ac^*(\X) \subset A^*(\X)$ the locally convex subspace of differential forms with compact support on $\X$, meaning,
\[
\Ac^*(\X)=\left\{\xi\in A^*(\X)  \;|\; \pi(\supp(\xi))\subset |\X| \mbox{ is compact}\right\}.
\]
Thus, for $f : \X \to \Y$ a proper morphism, the pull-back $f^* : \Ac^*(\Y) \to \Ac^*(\X)$ is well-defined. We need $f$ to be proper to guarantee it preserves compact support. If, on the other hand, $f$ is a relatively oriented submersion, the push-forward $f_* : \Ac^*(\Y) \to \Ac^*(\X)$ is well-defined even without the properness assumption, because the restriction to the support of the given form is always proper.

Denote by $\A^k(\X)$ the space of currents of cohomological degree $k$, that is, the continuous dual space of the compactly supported differential forms $\Ac^{\dim \X - k}(\X)$ with the weak* topology. When $\X$ is oriented, differential forms are identified as a subspace of currents by
\begin{gather*}
\varphi:A^k(\X)\hookrightarrow \A^k(\X),\\
\varphi(\gamma)(\eta)=\int_{\X}\gamma\wedge\eta,\quad \eta\in \Ac^{\dim \X -k}(\X).
\end{gather*}
The exterior derivative extends to currents via
$d\zeta(\eta)=(-1)^{1+|\zeta|}\zeta(d\eta).$ In particular, if $\d \X=\emptyset$, we have $d\varphi(\gamma) = \varphi(d\gamma).$
Currents are a bimodule over differential forms with
\[
(\eta\wedge \zeta)(\gamma)
:= (-1)^{|\eta|\cdot|\zeta|} \zeta(\eta\wedge\gamma),
\qquad
\gamma\in \Ac^*(\X),\quad \eta\in A^*(\X),\quad\zeta\in\A^*(\X),
\]
and
\[
(\zeta\wedge \eta)(\gamma)
:= \zeta(\eta\wedge\gamma),
\qquad
\gamma\in \Ac^*(\X), \quad \eta\in A^*(\X),\quad\zeta\in\A^*(\X).
\]
This bimodule structure makes $\varphi$ a bimodule homomorphism.

Let $f:\X\to \Y$ be a proper morphism in $\Orb$.
Set $\rdim f:=\dim\Y-\dim\X$.
Define the push-forward
\begin{equation}\label{eq:pfc}
f_* : \A^k(\X) \to \A^{k-\rdim f}(\Y)
\end{equation}
by
\[
(f_*\zeta)(\eta)=(-1)^{m\cdot \rdim f}\zeta(f^*\eta),\qquad \eta\in \Ac^m(\Y).
\]
So, when $f$ is a submersion, $f_* \varphi(\a) = \varphi(f_*\a).$
Similarly, for $f:\X\to \Y$ a relatively oriented submersion, define the pull-back
\begin{equation}\label{eq:pbc}
f^* : \A^{k}(\Y) \to \A^{k}(\X)
\end{equation}
by
\[
(f^*\zeta)(\eta)=\zeta(f_*\eta),\qquad \eta\in \Ac^m(\Y).
\]

We can now formulate an analogue of Theorem~\ref{thm:main} for currents:

\begin{prop}\label{prop:cur}\leavevmode
\begin{enumerate}[label=(\roman*)]
    \item\label{nt-cur}
Let $f,g : \X \to \Y$ be morphisms of orbifolds with corners and let $\alpha : f \Rightarrow g$ be a $2$-morphism. If $f,g,$ are proper, then $f_* = g_* : \A^k(\X) \to \A^{k-\rdim f}(\Y).$ If $f,g,$ are relatively oriented submersions and $\a$ is relatively oriented, then also $f^* = g^*: \A^*(\Y) \to \A^*(\X).$
	\item\label{opushcomp-cur}
Let $g: \X\to \Y$, $f:\Y\to \cZ,$ be proper morphisms of orbifolds with corners. Then
		\[
		(f_*\circ g_*)(\zeta)=(f\circ g)_*\zeta,\qquad \forall \zeta\in \A^*(\X).
		\]
If $f,g,$ are relatively oriented submersions, then
\[
(g^* \circ f^*)(\zeta) = (f\circ g)^*\zeta,\qquad \forall \zeta\in \A^*(\mZ).
\]
	\item\label{opushpull-cur}
		Let $f:\X\to \Y$ be a proper morphism of orbifolds with corners. Then
    \[
	f_*(f^*\eta\wedge\zeta)=\eta\wedge f_*\zeta, \qquad
\forall \,\eta\in A^*(\Y), \quad \zeta\in \A^*(\X).
    \]
If in addition $f$ is a relatively oriented submersion, then
	\[
	f_*(f^*\zeta\wedge\eta)=\zeta\wedge f_*\eta, \qquad
\forall \,\eta\in A^*(\X), \quad \zeta\in \A^*(\Y).
    \]
	\item\label{opushfiberprod-cur}
		Let
		\[
		\xymatrix{
		{\X\times_\cZ \Y}\ar[r]^{\quad p}\ar[d]^{q}&
        {\Y}\ar[d]^{g}\ar@2{L->L}[dl]\\
        {\X}\ar[r]^{f}&\cZ
		}
		\]
		be a weak pull-back diagram of orbifolds with corners,
where $f$ and $g$ are strongly smooth, $f$ is a relatively oriented submersion and $g$ is proper.
Equip $p$ with the transpose pull-back relative orientation and let $\zeta\in \A^*(\Y).$ Then
		\[
		q_*p^*\zeta=f^*g_*\zeta.
		\]
\end{enumerate}
\end{prop}

\begin{proof}
\leavevmode
\begin{enumerate}[label=(\roman*)]
\item
    By definition, $\rdim f=\rdim g$.
    By Theorem~\ref{thm:main}\ref{nt}, for all $\gamma\in \Ac^*(\X)$, we have
    \[
    (f_*\zeta)(\gamma)
    = (-1)^{|\gamma|\cdot \rdim f}\zeta(f^*\gamma)
    = (-1)^{|\gamma|\cdot \rdim g}\zeta(g^*\gamma)
    =(g_*\zeta)(\gamma).
    \]
Similarly, for all $\gamma\in \Ac^*(\Y)$,
    \[
    (f^*\zeta)(\gamma)
    =(-1)^{|\gamma|\cdot \rdim f}\zeta(f_*\gamma)
    =(-1)^{|\gamma|\cdot \rdim g}\zeta(g_*\gamma)
    =(g^*\zeta)(\gamma).
    \]
\item
    By Theorem~\ref{thm:main}\ref{opushcomp}, for all $\gamma\in\Ac^*(\X)$ we have
    \begin{multline*}
    \qquad
    (f\circ g)_*(\zeta)(\gamma)
    = (-1)^{|\gamma|\cdot \rdim(f\circ g)} \zeta((f\circ g)^*\gamma)
    = (-1)^{|\gamma|\cdot (\rdim f+\rdim g)} \zeta(g^*f^*\gamma)=\\
    = (-1)^{|\gamma|\cdot \rdim f} (g_*\zeta)(f^*\gamma)
    = (f_*g_*\zeta)(\gamma).
    \end{multline*}
    Similarly, for all $\gamma\in \Ac^*(\mZ)$,
    \[
    (f\circ g)^*(\zeta)(\gamma)
    = \zeta((f\circ g)_*\gamma)
    = g^*f^*\zeta(\gamma).
    \]
\item
    For all $\gamma \in\Ac^*(\Y)$, we have
    \begin{align*}
    f_*(f^*\eta\wedge\zeta)(\gamma)
    =& (-1)^{|\gamma|\cdot \rdim f} (f^*\eta\wedge\zeta)(f^*\gamma)\\
    =& (-1)^{|\gamma|\cdot\rdim f+|\eta|\cdot|\zeta|} \zeta(f^*\eta\wedge f^*\gamma)\\
    =& (-1)^{|\gamma|\cdot \rdim f +|\eta|\cdot|\zeta| + (|\eta|+|\gamma|)\cdot \rdim f} (f_*\zeta)(\eta\wedge \gamma)\\
    =& (-1)^{|\eta|\cdot|\zeta|+|\eta|\cdot \rdim f+|\eta|(|\zeta|-\rdim f)} (\eta\wedge f_*\zeta)(\gamma)\\
    =&(\eta\wedge f_*\zeta)(\gamma).
    \end{align*}
If $f$ is a relatively oriented submersion,
\begin{align*}
f_*(f^*\zeta\wedge\eta)(\gamma)
=& (-1)^{|\gamma|\cdot \rdim f} (f^*\zeta\wedge\eta)(f^*\gamma)\\
=& (-1)^{|\gamma|\cdot \rdim f}
    (f^*\zeta)(\eta\wedge f^*\gamma)\\
=& (-1)^{|\gamma|\cdot \rdim f}
    \zeta(f_*(\eta\wedge f^*\gamma))\\
=& (-1)^{|\gamma|\cdot \rdim f+|\eta||\gamma|}
    \zeta(f_*(f^*\gamma\wedge \eta))\\
=& (-1)^{|\gamma|\cdot \rdim f+|\eta||\gamma|}
    \zeta(\gamma\wedge f_*\eta)\\
=& (-1)^{|\gamma|\cdot \rdim f+|\eta||\gamma|+|\gamma||f_*\eta|}
    \zeta(f_*\eta\wedge\gamma)\\
=& (\zeta\wedge f_*\eta)(\gamma).
    \end{align*}
\item
    Using the orbifold analog of Lemma~\ref{lm:flip} and $\rdim g=\rdim q$, we get
  \begin{align*}
    (q_*p^*\zeta)(\gamma)
  =& (-1)^{|\gamma|\cdot\rdim q} (p^*\zeta)(q^*\gamma)\\
  =& (-1)^{|\gamma|\cdot\rdim q} \zeta(p_*q^*\gamma)\\
  =& (-1)^{|\gamma|\cdot\rdim q+\rdim{f}\cdot \rdim{g}}
    \zeta(g^*f_*\gamma)\\
  =& (-1)^{|\gamma|\cdot\rdim q +\rdim{f}\cdot \rdim{g} + |f_*\gamma|\cdot\rdim g} (f^*g_*\zeta)(\gamma)\\
  =& (f^*g_*\zeta)(\gamma).
  \end{align*}
\end{enumerate}
\end{proof}

We proceed to formulate a version of Theorem~\ref{thm:main}\ref{thmstokes} for currents. For this, we need a notion of restriction to the boundary. However, restricting a current to the boundary in general requires conditions on its wavefront set. To formulate the analogue of Theorem~\ref{thm:main}\ref{thmstokes} in elementary terms, we consider only currents that come from differential forms by the inclusion $\varphi.$ However, as noted above, $\varphi$ commutes with $d$ only if $\d \X=\emptyset$. Otherwise, a correction is needed that depends on the boundary, in the spirit of Stokes' theorem.  To avoid the issue, we introduce the following modification of the complex of currents.

Define differential forms relative to the boundary by
\[
\Ac^*(\X,\d\X)=\{\eta\in \Ac^*(\X)\,|\, i_{\X}^*\eta=0\},
\]
and denote by $\An^k(\X)$ the dual space of $\Ac^{\dim\X-k}(\X,\d\X)$. Consider the inclusion
\[
\psi: \A^*(\X)\lrarr \An^*(\X).
\]
This now produces an operation that commutes with $d$ even for orbifolds with non-empty boundary:
\begin{lm}
$d\circ (\psi\circ \varphi)= (\psi\circ \varphi)\circ d$.
\end{lm}

\begin{proof}
For $\eta\in A^*(\X)$ and $\gamma\in \Ac^*(\X,\d\X)$, we have
\begin{align*}
d\circ (\psi\circ\varphi)(\eta)(\gamma)
=&
(-1)^{1+|\eta|}(\psi\circ \varphi)(\eta)(d\gamma)\\
=&
(-1)^{1+|\eta|}\int_{\X} \eta\wedge d\gamma\\
=&
-\int_{\X} d(\eta\wedge \gamma)+\int_{\X} d\eta\wedge \gamma\\
=&
\int_{\X} d\eta\wedge \gamma\\
=& (\psi\circ\varphi)(d\eta)(\gamma).
\end{align*}
\end{proof}

The last thing we need before formulating a version of Stokes' theorem for currents is the following observation.

\begin{lm}\label{lm:horvan}
Let $f:M\to N$ be a strongly smooth map of manifolds with corners and
let $\gamma\in A^*(N)$ such that $i_N^*\gamma=0$. Then
$(\ih_M)^*f^*\gamma=0$.
\end{lm}

\begin{proof}
With the notation of Section~\ref{sssec:vhbd}, we have the following commutative diagram.
\[
\xymatrix{ \Xi \ar[r]^{\xi}\ar[d]^{\pi_2}& \dh M \ar[r]^{\ih_M} & M \ar[d]^{f} \\
\partial N \ar[rr]^{i_N} && N
}
\]
By~\cite[Proposition 4.2]{Joyce2}, $\Xi$ can be given the structure of a smooth manifold with corners such that $\xi$ is a covering map and $\pi_2$ is smooth.
By assumption, $\pi_2^*i_N^*\gamma=0$, and therefore $\xi^*(\ih_M)^*f^*\gamma=0$.  Since $\xi$ is a surjective submersion, it follows that $(\ih_M)^*f^*\gamma=0$.
\end{proof}

\begin{rem}
In the situation of Lemma~\ref{lm:horvan}, it is not true in general that $(i^v_M)^*f^*\gamma = 0.$ Thus pull-back of differential forms does not preserve differential forms relative to the boundary. So, it is better to work with usual differential forms when defining pull-back and push-forward.
\end{rem}

In the following, we consider only oriented orbifolds so that the inclusion $\varphi : A^*(\X) \to \A^*(\X)$ is defined. More generally, one should consider differential forms with coefficients in a local system.
Denote by $\varphi_{\dv}$ the inclusion
\[
\varphi_{\dv}: A^*(\d\X)\lrarr \A^*(\dv \X).
\]

\begin{prop}
Let $f:\X\to \Y$ be a strongly smooth proper morphism
of oriented orbifolds with corners with $\dim \X=s$, and let
$\eta\in A^t(\X)$.
Then
\[
\psi(f_*(\varphi(d\eta))) = d\psi(f_*(\varphi(\eta))) +(-1)^{s+t}  \psi((f\circ \iv_{\X})_*(\varphi_{\dv})((\iv_{\X})^*\eta)).
\]
\end{prop}
\begin{proof}
For $\gamma\in \Ac^{\dim \X -|\eta|-1}(\Y,\d\Y)$, compute
\begin{align*}
\psi(d(f_*\varphi(\eta)))(\gamma)
    =& (-1)^{1+|\eta|+\rdim f} (f_*\varphi(\eta))(d\gamma)\\
    =& (-1)^{1+|\eta|+\rdim f +\rdim f\cdot(1+|\gamma|)} \varphi(\eta)(f^*d\gamma)\\
    =& (-1)^{1+|\eta|+\rdim f\cdot|\gamma|} \int_{\X}\eta\wedge d(f^*\gamma)\\
    =& (-1)^{1+\rdim f\cdot |\gamma|}
    \big(\int_{\X} d(\eta\wedge f^*\gamma)
    - \int_{\X} d\eta\wedge f^*\gamma \big) \\
    =&(-1)^{1+\rdim f\cdot|\gamma|}
     \big(\int_{\d \X} (\eta\wedge f^*\gamma)
    -\int_{\X} d\eta\wedge f^*\gamma \big).\\
\shortintertext{By Lemma~\ref{lm:horvan},}
    =&(-1)^{1+\rdim f\cdot|\gamma|}
     \big(\int_{\dv \X} (\eta\wedge f^*\gamma)
    -\int_{\X} d\eta\wedge f^*\gamma \big)\\
    =&(-1)^{1+\rdim f\cdot|\gamma|}
     \big(\varphi_{\dv}(\eta|_{\dv\X})(f^*\gamma|_{\dv\X})
    -\varphi(d\eta)(f^*\gamma) \big)\\
    =&(-1)^{1+\rdim f\cdot |\gamma|+ |\gamma|\cdot (\rdim f+1)}
     (f|_{\dv\X})_*\varphi_{\dv}(\eta|_{\dv\X})(\gamma)-\\
    &\quad -(-1)^{1+\rdim f\cdot|\gamma| + |\gamma|\cdot \rdim f} f_*\varphi(d\eta)(\gamma)\\
    =&(-1)^{1+|\gamma|}
     (f|_{\dv\X})_*\varphi_{\dv}(\eta|_{\dv\X})(\gamma)
    +f_*\varphi(d\eta)(\gamma)\\
    =&(-1)^{\dim\X+|\eta|}
     (f|_{\dv\X})_*\varphi_{\dv}(\eta|_{\dv\X})(\gamma)
    +f_*\varphi(d\eta)(\gamma).
\end{align*}
\end{proof}

\bibliography{../../bibliography_exp}
\bibliographystyle{../../amsabbrvcnobysame}
\end{document}